\documentclass[letterpaper,12pt]{amsart}

\usepackage{amssymb}
\usepackage{amsmath,amscd}

\numberwithin{equation}{section}

\newtheorem{theorem}{Theorem}[section]
\newtheorem{conjecture}[theorem]{Conjecture}
\newtheorem{question}[theorem]{Question}
\newtheorem{hypothesis}[theorem]{Hypothesis}
\newtheorem{lemma}[theorem]{Lemma}
\newtheorem{proposition}[theorem]{Proposition}
\newtheorem{definition}[theorem]{Definition}
\newtheorem{corollary}[theorem]{Corollary}

\theoremstyle{definition}

\newtheorem{remark}[theorem]{Remark}

\newcommand{\Gal}{\mathrm{Gal}}
\newcommand{\Hol}{\mathrm{Hol}}

\newcommand{\rad}{\mathrm{rad}}
\newcommand{\PSL}{\mathrm{PSL}}
\newcommand{\PSU}{\mathrm{PSU}}
\newcommand{\PSp}{\mathrm{PSp}}
\newcommand{\SL}{\mathrm{SL}}
\newcommand{\GL}{\mathrm{GL}}
\newcommand{\Aut}{\mathrm{Aut}}

\newcommand{\Aff}{\mathrm{Aff}}

\newcommand{\Out}{\mathrm{Out}}

\newcommand{\id}{\mathrm{id}}
\newcommand{\F}{\mathbb{F}}

\newcommand{\core}{\mathrm{core}}
\newcommand{\soc}{\mathrm{soc}}
\newcommand{\onto}{\twoheadrightarrow}
\newcommand{\bs}{\backslash}

\newcommand{\hA}{\widehat{A}}
\newcommand{\hB}{\widehat{B}}
\newcommand{\hC}{\widehat{C}}
\newcommand{\hD}{\widehat{D}}
\newcommand{\hM}{\widehat{M}}
\newcommand{\hP}{\widehat{P}}
\newcommand{\Tr}{\mathrm{Trans}}

\newcommand{\NN}{\mathcal{N}}
\newcommand{\MM}{\mathcal{M}}
\newcommand{\JJ}{\mathcal{J}}

\begin{document}


\title[On Insoluble Transitive Subgroups]{On Insoluble Transitive Subgroups in the Holomorph of a Finite Soluble Group}

\author{Nigel P.~Byott}
\address{Department of Mathematics \& Statistics, University of Exeter, Exeter 
EX4 4QF, U.K.}  
\email{N.P.Byott@exeter.ac.uk}

\date{\today}

\thanks{This work was supported by the Engineering and Physical Sciences Research Council [grant number EP/V005995/1]. \newline 
\indent The author thanks Senne Trappeniers for pointing out a significant error in an earlier version of this paper. The author also thanks the referee for their detailed comments and suggestions. These have considerably improved the paper.
\newline
\indent
For the purpose of open access, the author has applied a CC BY public copyright licence  to any Author Accepted Manuscript version arising. \newline 
\indent
Data Access Statement: Data sharing is not applicable to this article as no datasets were generated or analysed in this research.}

\bibliographystyle{amsalpha}

\begin{abstract}
A question of interest both in Hopf-Galois theory and in the theory of skew braces is whether the holomorph $\Hol(N)$ of a finite soluble group $N$ can contain an insoluble regular subgroup. We investigate the more general problem of finding an insoluble transitive subgroup $G$ in $\Hol(N)$ with soluble point stabilisers. We call such a pair $(G,N)$ irreducible if we cannot pass to proper non-trivial quotients $\overline{G}$, $\overline{N}$ of $G$, $N$ so that 
$\overline{G}$ becomes a subgroup of $\Hol(\overline{N})$. We classify all irreducible solutions $(G,N)$ of this problem, showing in particular that every non-abelian composition factor of $G$ is isomorphic to the simple group of order $168$. 
Moreover, every maximal normal subgroup of $N$ has index $2$. 
\end{abstract} 

\maketitle

\section{Introduction}

This paper is motivated by the following group-theoretic conjecture, which has arisen both in Hopf-Galois theory \cite[end of \S1]{HGSsol} and in the theory of skew braces \cite[Problem 2.46]{V-probs}.

\begin{conjecture} \label{main-conj}
If $N$ is a finite soluble group then any regular subgroup $G$ in the holomorph $\Hol(N)$ of $N$ is also soluble.
\end{conjecture}

Tsang and Qin \cite{TQ} showed computationally that there are no counterexamples with $|N| \leq 2000$, and showed theoretically that there are no counterexamples with $|N|$ cube-free or of the form $2^r |S|$ for certain non-abelian simple groups $S$.  Gorshkov and Nasybullov \cite{GN} proved that, in a counterexample of minimal order, the insoluble subgroup $G$ cannot be a simple group. (This can also be deduced from \cite{simple}.) 
Nasybullov \cite{Nasybullov} has also shown that the analogue of Conjecture \ref{main-conj} for infinite groups is false. Reversing the roles of $G$ and $N$ in Conjecture \ref{main-conj}, the holomorph of a finite insoluble group may contain a soluble regular subgroup \cite{HGSsol}.
 
In this paper, we provide further evidence in support of Conjecture \ref{main-conj}. We consider minimal counterexamples to Conjecture \ref{main-conj} in a sense made precise in Definition \ref{minimal}. 
We show in Theorem \ref{main-thm} that if the pair of groups $(G,N)$ is a minimal counterexample then $G$ has a very special form: it contains a soluble normal subgroup $K$ such that, for some $r \geq 1$ and some soluble transitive subgroup $H$ of the symmetric group $S_r$, the quotient group $G/K$ is isomorphic to the wreath product $\GL_3(2) \wr H$ of the non-abelian simple group $\GL_3(2)=\PSL_3(2) \cong \PSL_2(7)$ of order $168$ by $H$. Thus $\GL_3(2)$ is the unique non-abelian composition factor of $G$, and it occurs with multiplicity $r$. 

Since transitive subgroups of $\Hol(N)$ exhibit better behaviour under quotients than regular subgroups (see Proposition \ref{quot} below), we approach Conjecture  \ref{main-conj} by relaxing the requirement for regularity. We therefore ask:

\begin{question}  \label{main-qn}
	What are the pairs $(G,N)$ of finite groups such that $N$ is soluble, $G$ is an insoluble transitive subgroup of $\Hol(N)$, and the stabiliser in $G$ of any element of $N$ is also soluble?
\end{question}

It is perhaps surprising that such pairs $(G,N)$ do exist. In the smallest example, $N$ is the elementary abelian group of order $8$ and $G \cong \GL_3(2)$. We will define a notion of irreducibility for a solution $(G,N)$ of Question \ref{main-qn}, and then classify all irreducible solutions. More precisely, in Theorem \ref{irred-constr} we describe a family of irreducible solutions, and in Theorem 
\ref{irred-class} we show that there are no others. Theorem \ref{main-thm} is then a consequence of Theorem \ref{irred-class}. 

For the reader's convenience, we now briefly review the topics of Hopf-Galois theory and skew braces which form the background to this work. 

We first discuss Hopf-Galois theory.  A more detailed exposition can be found in \cite[Chapter 2]{Ch00}. If $L/K$ is a finite Galois extension of fields with Galois group $G$ then there may be many $K$-Hopf algebras $H$ which act on $L$ to make $L$ into an $H$-Galois extension. Greither and Pareigis \cite{GP} showed that these Hopf-Galois structures correspond to subgroups $N$ in the symmetric group on $G$ which act regularly on $G$ and are normalised by left translations. The Hopf-Galois structures can therefore be classified according to the isomorphism type of $N$, which is called the type of the Hopf-Galois structure. The extension $L/K$ admits a Hopf-Galois structure of type $N$ if and only if $\Hol(N)$ contains a regular subgroup isomorphic to $G$. Thus Conjecture \ref{main-conj} is equivalent to the assertion that a finite Galois extension with insoluble Galois group cannot admit a Hopf-Galois structure of soluble type. More generally, a solution $(G,N)$ to Question \ref{main-qn} corresponds to a
Hopf-Galois structure of soluble type $N$ on a separable (but not necessarily normal) field extension $L/K$  with normal closure $E/K$ such that $G=\Gal(E/K)$ is insoluble but $\Gal(E/L)$ is soluble. 

We now turn to skew braces. Braces were introduced by Rump \cite{Rump} in order to 
study non-degenerate involutive set-theoretical solutions of the Yang-Baxter equation.  Guarnieri and Vendramin \cite{GV} defined the more general concept of skew braces to study non-degenerate solutions of the Yang-Baxter equation which are not necessarily involutive. A skew brace  
$(B,+, \circ)$ consists of a set $B$ and two operations, each making $B$ into a group, and satisfying a certain compatibility condition by virtue of which the multiplicative group $(B,\circ)$ has an action on the additive group $(B,+)$. A skew brace is a brace if $(B,+)$ is abelian. As explained in the introduction 
to \cite{bachiller}, if $(B,+,\circ)$ is a finite brace then $(B,\circ)$ is soluble. In any skew brace, $(B, \circ)$ can be viewed as a regular subgroup of the holomorph of $(B,+)$, and conversely a regular subgroup $G$ in the holomorph of a group $(B,+)$ gives rise to a skew brace $(B,+,\circ)$ with $(B,\circ) \cong G$. Thus Conjecture \ref{main-conj} is equivalent to the assertion that any finite skew brace with soluble additive group must also have soluble multiplicative group. 

We end this introduction with an outline of the contents of the paper. In \S\ref{admis}, we describe how our problem behaves under passing to subgroups and quotient groups of $N$. In \S\ref{statements} we formulate our main results Theorems \ref{main-thm}, \ref{irred-constr} and \ref{irred-class}, and we deduce Theorem \ref{main-thm} from Theorem \ref{irred-class}. In \S\ref{example-sec}, we construct some irreducible solutions to Question \ref{main-qn}, thereby proving Theorem \ref{irred-constr}. The remainder of the paper is devoted to the proof of Theorem \ref{irred-class}. In any irreducible solution of Question \ref{main-qn}, $N$ is an elementary abelian group. Thus $N$ may be identified with a vector space $V=\F_p^n$ for some prime $p$ and some $n \geq 1$, and $\Hol(N)$ then becomes the affine group $\Aff(V)$ of $V$. In \S\ref{affine} we give some general results on affine groups of finite vector spaces. We then deduce some information about the structure of the group $G$ in \S\ref{subgps-affine}, in particular showing that each composition factor of $G$ contains a soluble subgroup of $p$-power index. In \S\ref{clifford} we decompose $V$ as a module over the socle $S$ of $G$. This leads to an inequality (\ref{key-ineq}) relating the orders of the non-abelian composition factors of $S$ and their automorphism groups to various numerical parameters describing the structure of $G$ and $N$. Our conclusions up to this point are collected together in Theorem \ref{no-CFSG}. To proceed further, we apply the Classification of Finite Simple Groups. This is done in \S\ref{cfsg} via Guralnick's determination \cite{Gural} of the non-abelian simple groups with a subgroup of prime-power index. We obtain in Corollary \ref{simple-sol-p} a list of the possible non-abelian composition factors of $S$ and the corresponding primes $p$. Combining this with (\ref{key-ineq}), we prove Theorem \ref{just-168}, which eliminates all potential non-abelian composition factors except the simple group of order $168$. Theorem \ref{just-168} also shows that $p$ must be $2$. Finally, we show in \S\ref{conclusion} that $G/S$ is soluble and the pair $(G,N)$ can only be one of those described in Theorem \ref{irred-constr}. 
 
\section{$G$-admissible and $G$-invariant subgroups}  \label{admis}

The holomorph of a group $N$ is by definition the semidirect product $\Hol(N)=N \rtimes \Aut(N)$. 
We write its elements as $(\alpha, \theta)$ with $\alpha \in N$ and $\theta \in \Aut(N)$, 
so the group operation on $\Hol(N)$ is given by 
\begin{equation} \label{hol-mult}
	(\alpha, \theta)(\beta, \phi) = (\alpha \theta(\beta), \theta \phi),
\end{equation}
and $\Hol(N)$ acts as permutations of $N$ by
\begin{equation} \label{hol-action}
	(\alpha, \theta) \cdot \eta = \alpha \theta(\eta). 
\end{equation}

We say that a subgroup $G$ of $\Hol(N)$ is transitive (respectively, regular) if, for all $\eta$, $\mu \in N$, there is a $g \in G$ (respectively, a unique $g \in G$) with $g \cdot \eta = \mu$. 

Now let $N$ be an arbitrary finite group, and let $G$ be a transitive subgroup of $\Hol(N)$.  We  describe conditions under which $G$ gives rise to a transitive subgroup of the holomorph of a subgroup or quotient group of $N$. We frequently write $g \in G$ as $(\alpha_g, \theta_g)$ with $\alpha_g \in N$ and $\theta_g \in \Aut(N)$. Then $\alpha_g= g \cdot e$, where $e$ is the identity element of $N$.

\begin{definition} \label{ad-invt}
Let $M$ be a subgroup of $N$. 
\begin{itemize}
\item[(i)] $M$ is {\em $G$-admissible} if the subset
$$ M_*  = \{ g \in G : \alpha_g \in M\} $$
of $G$ is a subgroup of $G$.
\item[(ii)] $M$ is {\em $G$-invariant} if $\theta_g(m) \in M$ for all $g \in G$ and $m \in M$.
\end{itemize}
\end{definition}

\begin{proposition}  \label{M-star}
Let $G$ be a transitive subgroup of $\Hol(N)$. For a subgroup $M$ of $N$, the following are equivalent:
\begin{itemize}
\item[(i)] $M$ is $G$-admissible;
\item[(ii)] $g \cdot m \in M$ for all $g \in M_*$ and all $m \in M$;
\item[(iii)] $\theta_g(m) \in M$ for all $g \in M_*$ and all $m \in M$.
\end{itemize}
\end{proposition}
\begin{proof}
If $g \in M_*$ and $m \in M$ then $g \cdot m =\alpha_g \theta_g(m)$ with $\alpha_g \in M$, so (ii) and (iii) are equivalent. 
Assume that these conditions hold. Then, for $g$, $h \in M_*$ we have $(gh) \cdot e = g \cdot (h \cdot e)=g \cdot \alpha_h \in M$ since $\alpha_h \in M$.
Thus $M_*$ is closed under multiplication. Since $M_*$ is finite and contains the identity element $(e, \id)$ of $\Hol(N)$, it follows that $M_*$ is a subgroup of $G$.
Hence (i) holds. 

Conversely, assume (i) and let $g \in M_*$ and $m \in M$. As $G$ is transitive, there is some $h \in G$ with $h \cdot e =m$. Then $h \in M_*$ and $\alpha_h=m$. Since $M$ is $G$-admissible, we have $gh \in M_*$ and so $g \cdot m = g \cdot (h \cdot e)= (gh) \cdot e \in M$ since $gh \in M_*$. Thus (ii) holds.
\end{proof}

\begin{remark} \label{inv-ad-rk}
It follows from Proposition \ref{M-star} that if $M$ is $G$-invariant then $M$ is $G$-admissible. In the other direction, if $M$ is $G$-admissible and $H=M_*$ then $M$ is $H$-invariant (but not necessarily $G$-invariant). 
\end{remark}

\begin{lemma} \label{admis-sub}
Let $G$ be a transitive subgroup of $\Hol(N)$, let $M$ be a $G$-admissible subgroup of $N$, and let $H=M_*$. Then the action of $G$ on $N$ restricts to an action of $H$ on $M$. In particular, if $G$ is a regular subgroup of $\Hol(N)$ then $H$ is a regular subgroup of $\Hol(M)$. Conversely, suppose that $G$ is a regular subgroup of $\Hol(N)$ and $H$ is a subgroup of $G$ acting regularly on some subgroup $M$ of $N$. Then $M$ is $G$-admissible and $H=M_*$.
\end{lemma}
\begin{proof} 
For the first statement, if $h=(\alpha_h,\theta_h) \in H$ then $\alpha_h \in M$, and $\theta_h$ restricts to an automorphism of $M$ by Proposition \ref{M-star}. Thus for $m \in M$ we have $h \cdot m= \alpha_h \theta_h(m) \in M$. We may therefore view $h$ as an element of $\Hol(M)$. Since $G$ is transitive on $N$, it is clear that $H$ is transitive on $M$. If $G$ is regular on $N$ then the stabiliser of $e$ in $H$ is trivial, so $H$ is regular on $M$. 

For the converse, if $G$ is regular on $N$ and $H$ acts regularly on some subgroup $M$ then we must have $H=\{ g \in G: g \cdot e \in M \} = M_*$. Hence $M_*$ is a group and $M$ is $G$-admissible.
\end{proof}

\begin{remark}
In the language of braces, Lemma \ref{admis-sub} says that if $G$ is regular on $N$ (so that $N$ is a skew brace) then $M$ is a sub-skew brace of $N$ if and only if $M$ is a $G$-admissible subgroup of $N$.
\end{remark}

We next consider quotients of $N$.

\begin{proposition}  \label{quot}
Let $G$ be a transitive subgroup of $\Hol(N)$, and let $M$ be a normal subgroup of $N$. Then the action of $G$ on $N$ induces an action of $G$ on 
$N/M$ if and only if $M$ is a $G$-invariant subgroup of $N$. In this case, $M$ is $G$-admissible. Moreover, there is a normal subgroup
$K$ of $G$ such that $G/K$ embeds as a transitive subgroup of $\Hol(N/M)$.
\end{proposition}
\begin{proof}
The action of $G$ on $N$ induces a well-defined action on $N/M$ if and only if $g \cdot (nm) \in  (g \cdot n) M$ for all $g \in G$, $n \in N$ and $m \in M$. This is equivalent to $\alpha_g \theta_g(n) \theta_g(m) \in \alpha_g \theta_g(n)M$ for all $g$, $n$, $m$, and thus to the $G$-invariance of $M$. If $M$ is a $G$-invariant normal subgroup, then $M$ is $G$-admissible by Remark \ref{inv-ad-rk}, and 
the stabiliser of the trivial coset $eM$ is 
$\{g \in G : g \cdot e \in M\}=M^*$.
Now let $K$ be the core of $M^*$ in $G$:
$$ K=\core_G(M^*) = \bigcap_{g \in G} g M^* g^{-1}.  $$
Then $K \lhd G$, and $G/K$ embeds as a subgroup of $\Hol(N/M)$. Moreover $G/K$ is transitive on $N/M$ since $G$ is transitive on $N$. 
\end{proof}

\section{Statement of the main results} \label{statements}

In this section, we use the ideas developed above to formulate our main results on Conjecture \ref{main-conj} and Question \ref{main-qn}. 

\begin{definition}  \label{minimal}
We call a pair of groups $(G,N)$ a {\em counterexample to Conjecture \ref{main-conj}} if $N$ is a finite soluble group and $G$ is an insoluble regular subgroup of $\Hol(N)$. The pair $(G,N)$ is a {\em minimal 
counterexample to Conjecture \ref{main-conj}} if, in addition, for every proper $G$-admissible subgroup 
$M$ of $N$, the subgroup $M_*$ of $G$ is soluble.
\end{definition}

Our main result on Conjecture \ref{main-conj} is the following.

\begin{theorem} \label{main-thm} \ 
Let $(G,N)$ be a minimal counterexample to Conjecture \ref{main-conj}. Then every maximal normal subgroup of 
$N$ has index $2$, and $G$ has a soluble normal subgroup $K$ such that, for some $r \geq 1$ and some transitive soluble subgroup $H$ of the symmetric group $S_r$, we have 
	$$ G/K \cong \GL_3(2) \wr H = \GL_3(2)^r \rtimes H, $$
	the wreath product of $\GL_3(2)$ by $H$.
\end{theorem}

Theorem \ref{main-thm} will be deduced from Theorem \ref{irred-class} below at the end of this section. 

\begin{corollary} \label{min-subquot}
 If $(G,N)$ is any counterexample to Conjecture \ref{main-conj} then the simple group $\GL_3(2)$ of order $168$ occurs as a subquotient of $G$.
\end{corollary}
\begin{proof}
If $(G,N)$ is not a minimal counterexample, we may replace it by a pair $(H,M)$ where $M$ is a proper $G$-admissible subgroup of $N$ and $H=M_*$ is insoluble. Repeating this construction, we eventually reach a minimal counterexample $(G_1,N_1)$ in which $G_1$ is a subgroup of $G$. By Theorem \ref{main-thm}, $\GL_3(2)$ occurs as a composition factor of $G_1$, and therefore as a subquotient of $G$. 
\end{proof}

\begin{definition}  \label{irred}
For a pair of finite groups $(G,N)$ with $G \leq \Hol(N)$, we will say that $(G,N)$ is {\em irreducible} if $N$ has no non-trivial proper $G$-invariant normal subgroup. 
\end{definition}

\begin{lemma} \label{irred-el-ab}
Let $(G,N)$ be irreducible, with $N$ soluble. If $G$ acts transitively on $N$ then $N$ is an elementary abelian $p$-group for some prime $p$.
\end{lemma}
\begin{proof}
Since $N$ has no non-trivial proper $G$-invariant normal subgroup, it has no non-trivial proper characteristic subgroup. Thus $N$ is characteristically simple, so $N$ is the direct product of isomorphic simple groups by \cite[3.3.15]{Robinson}. Since $N$ is soluble, the conclusion follows. 
\end{proof}

Perhaps surprisingly, irreducible solutions of Question \ref{main-qn} do exist.

\begin{theorem} \label{irred-constr}
Let $r \geq 1$, and let $H$ be a transitive soluble subgroup of the symmetric group $S_r$ of degree $r$. Let $N$ be an elementary abelian group of order $2^{3r}$. Then $\Hol(N)$ contains a transitive subgroup $G$ such that $ G \cong  \GL_3(2) \wr H$, and $N$ contains no non-trivial proper $G$-admissible subgroup. In particular, $(G,N)$ is an irreducible solution of Question \ref{main-qn}.
\end{theorem}

An explicit construction for the group $G$ in Theorem \ref{irred-constr} will be given in \S\ref{example-sec}.

The central result of this paper is that Theorem \ref{irred-constr} gives all the irreducible solutions. 

\begin{theorem}  \label{irred-class}
If $(G,N)$ is any irreducible solution to Question \ref{main-qn} then 
$(G,N)$ is as described in Theorem \ref{irred-constr}. In particular, $N$ is an elementary abelian $2$-group, any non-abelian composition factor of $G$ is isomorphic to the simple group $\GL_3(2)$ of order $168$, and $N$ contains no non-trivial proper $G$-admissible subgroup. Moreover, given $r$ and $H$ as in Theorem \ref{irred-constr}, the corresponding subgroup $G$ of $\Hol(N)$ is unique up to conjugation by $\Aut(N)$. 
\end{theorem}

\begin{proof}[Proof of Theorem \ref{main-thm} assuming Theorem \ref{irred-class}]
Let $(G,N)$ be a 
minimal counterexample to Conjecture \ref{main-conj}. Let $p$ be a prime such that $N$ has a maximal normal subgroup of index $p$, and let $L$ be the intersection of all such subgroups. Then $L$ is a proper characteristic subgroup of $N$, and $N/L$ is an elementary abelian $p$-group. In particular, $L$ is a $G$-invariant normal subgroup of $N$. Let $M$ be a maximal proper $G$-invariant normal subgroup of $N$ containing $L$, and let $\overline{N}=N/M$. By Proposition \ref{quot} there is a normal subgroup $K$ of $G$ so that $\overline{G}=G/K$ is a transitive subgroup of $\Hol(\overline{N})$. The stabiliser in $\overline{G}$ of the identity in $\overline{N}$ is $M_*/K$ where $M_*$ is soluble because $(G,N)$ is 
minimal. Hence $K$ is soluble. Since $G$ is insoluble, it follows that $\overline{G}$ is insoluble. Thus $(\overline{G},\overline{N})$ is a solution to Question \ref{main-qn}, and it is irreducible by the maximality of $M$. By Theorem \ref{irred-class}, $(\overline{G},\overline{N})$ is then  as described in Theorem \ref{irred-constr}. 
In particular, $|\overline{N}|=2^{3r}$ for some $r \geq 1$, so $p=2$. 
\end{proof}

\section{Some solutions to Question \ref{main-qn}} \label{example-sec}

Our main goal in this section is to prove Theorem \ref{irred-constr}. First we give an alternative description of $\Hol(N)$ when $N$ is an elementary abelian group, say of order $p^n$ for some prime $p$ and some $n \geq 1$. By Lemma \ref{irred-el-ab}, this situation arises for any irreducible solution $(G,N)$ to Question \ref{main-qn}. We identify $N$ with the vector space $V=\F_p^n$ of column vectors over the field $\F_p$ of $p$ elements. We write $V$ additively and denote its identity element $e$ by $0_V$. The group $\Aut(N)$ is then identified with the group $\GL_n(p)$ of invertible $n \times n$ matrices over $\F_p$, and $\Hol(N)=N \rtimes \Aut(N)$ is identified with the affine group $\Aff(V) = V \rtimes \GL_n(p)$ of $V$. 

For any subgroup $G$ of $\Aff(V)$, the natural homomorphism $G \to \Aut(V)=\GL_n(p)$ then makes $V$ into an 
$\F_p[G]$-module. We shall refer to the corresponding action of $G$ as the {\em linear action}, and the original action of $G$ as a subgroup of $\Aff(V)$ as the {\em affine action}.
Thus, in a solution to Question \ref{main-qn}, $G$ is transitive on $V$ in the affine action by hypothesis, but is not transitive in the linear action since $\{0_V\}$ is an orbit of cardinality $1$.  

Changing terminology, we will refer to subspaces $W$ of $V$ rather than to subgroups $M$ of $N$. Then $M$ is a $G$-invariant subgroup if and only if the corresponding subspace $W$ is an $\F_p[G]$-submodule of $V$ in the linear action. Moreover, $(G,V)$ is irreducible in the sense of Definition \ref{irred} if and only if $V$ is an irreducible $\F_p[G]$-module.

We next describe a convenient notational device for working with $\Aff(V)$.  An element of $\Aff(V)$ has the form $(v,A)$ for $v \in V$ and $A \in \GL_n(p)$. 
 We write this element as a block matrix of size $n+1$: 
$$ \left( \begin{array}{c|c} A & v \\ \hline 0 & 1 \end{array}, \right), $$
where the lower left entry is a row vector. Then matrix multiplication coincides with the multiplication
$$ (v,A) (w,B)=(v+Aw, AB) $$
in $\Hol(V)$ as specified in (\ref{hol-mult}). 

We now turn to the proof of Theorem \ref{irred-constr}. Let 
$T=\GL_3(2)$, the non-abelian simple group of order $168$. We begin by constructing a pair $(G,V)$ as in Question \ref{main-qn} with $G \cong T$ and $V=\F_2^3$. Thus we seek a transitive embedding of $T=\Aut(V)$ into 
$\Aff(V)$, say taking each $M \in T$ to 
\begin{equation} \label{def-psi}
   \hM = \left( \begin{array}{c|c} M & \psi(M) \\ \hline 0 & 1 \end{array} \right).
\end{equation}
for some function $\psi : T \to \F_2^3$. Then $\psi$ must satisfy the cocycle condition $\psi(M_1 M_2) = \psi(M_1) + M_1 \psi(M_2)$ for all $M_1$, $M_2 \in T$. We require $\psi$ to be surjective, so the stabiliser $G'$ of $0_V$ will be the image of a subgroup $T'$ of $T$ of index $8$. The image of any Sylow $2$-subgroup $P$ of $T$ will therefore be a complement to $G'$ in $G$, and so will act regularly on $V$. Now $T$ contains $8$ subgroups of index $8$, namely the normalisers of the $8$ Sylow $7$-subgroups. We arbitrarily choose $T'$ to be generated by the matrices  
\begin{equation} \label{AB}
   A=\left(\begin{array}{ccc} 0 & 0 & 1\\ 1 & 0 & 1 \\
   0 & 1 & 0 \end{array} \right), \qquad
 B=\left(\begin{array}{ccc} 1 & 0 & 0  \\ 0 & 0 & 1  \\
   0 & 1 & 1  \end{array} \right)
\end{equation} 
and $P$ to be generated by the matrices 
$$ C=\left(\begin{array}{ccc} 1 & 1 & 0 \\ 0 & 1 & 1  \\
   0 & 0 & 1  \end{array} \right), \qquad
 D=\left(\begin{array}{ccc} 1 & 0 & 0  \\ 0 & 1 & 1  \\
   0 & 0 & 1  \end{array} \right).  $$
Then $A$, $B$, $C$, $D$ satisfy the relations 
\begin{equation} \label{AB-relation}
  A^7= B^3=I, \qquad 
      BA=A^2 \, B,
\end{equation}
\begin{equation} \label{CD-relation}
    C^4=D^2=I \neq C^2, \qquad
     DCD^{-1}= C^3,
\end{equation}
and also
\begin{equation} \label{ABCD-relation} 
  CA=A^3 C^2 D, 
   \quad CB = B^2 C^2 D, 
   \quad  DA = A C^3 D, 
   \quad  DB = B^2 D. 
\end{equation}
The image of $T$ in $\Aff(V)$ is then generated by
\begin{equation} \label{hA-hB}
 \hA=\left(\begin{array}{ccc|c} 0 & 0 & 1 & 0 \\ 1 & 0 & 1 & 0 \\
   0 & 1 & 0 & 0 \\ \hline 0 & 0 & 0 & 1 \end{array} \right), \qquad
 \hB=\left(\begin{array}{ccc|c} 1 & 0 & 0 & 0 \\ 0 & 0 & 1 & 0 \\
   0 & 1 & 1 & 0 \\ \hline 0 & 0 & 0 & 1 \end{array} \right) 
\end{equation}
and 
$$ \hC=\left(\begin{array}{ccc|c} 1 & 1 & 0 & c_1 \\ 0 & 1 & 1 & c_2 \\
   0 & 0 & 1 & c_3 \\ \hline 0 & 0 & 0 & 1 \end{array} \right), \qquad
 \hD=\left(\begin{array}{ccc|c} 1 & 0 & 0 & d_1 \\ 0 & 1 & 1 & d_2 \\
   0 & 0 & 1 & d_3 \\ \hline 0 & 0 & 0 & 1 \end{array} \right) $$ 
for some $c_1$, $c_2$, $c_3$, $d_1$, $d_2$, $d_3 \in \F_2$. 
These must be chosen so that the subgroup $G=\langle \hA, \hB, \hC, \hD \rangle$ of $\Aff(V)$ acts transitively on $V$ and has order no larger than $168$. Thus (\ref{AB-relation}), (\ref{CD-relation}) and (\ref{ABCD-relation}) must remain valid when $A$, $B$, $C$, $D$ are replaced by $\hA$, $\hB$, $\hC$, $\hD$. The relation 
$\hD \hB=\hB^2 \hD$ gives $d_2=d_3=0$. Then $d_1=1$ since $\hD$ is not in the stabiliser of $0_V$. The relation 
$\hD \hA=\hA \hC^3 \hD$ then gives $c_1=0$, $c_2=c_3=1$.  
Thus we have 
$$ \hC=\left(\begin{array}{ccc|c} 1 & 1 & 0 & 0 \\ 0 & 1 & 1 & 1 \\
   0 & 0 & 1 & 1 \\ \hline 0 & 0 & 0 & 1 \end{array} \right), \qquad
 \hD=\left(\begin{array}{ccc|c} 1 & 0 & 0 & 1 \\ 0 & 1 & 1 & 0 \\
   0 & 0 & 1 & 0 \\ \hline 0 & 0 & 0 & 1 \end{array} \right). $$ 

One can verify by direct calculation that these matrices $\hA$, $\hB$, $\hC$, 
$\hD$ do indeed satisfy the analogues of (\ref{AB-relation}), (\ref{CD-relation}) and (\ref{ABCD-relation}). In particular, the groups $G'=\langle \hA, \hB \rangle$ and 
$\hP=\langle \hC, \hD \rangle$ have orders $21$ and $8$ respectively, and $G=G' \hP = \hP G'$. Thus $|G|=168$ and $G \cong T$. Moreover, the stabiliser of $0_V$ contains the maximal subgroup $G'$ of $G$, but does not contain $\hC$, so $G$ acts transitively on $V$. 

We summarise what we have just shown.

\begin{lemma}  \label{unique-168}
Let $V=\F_2^3$. Then $\Aff(V)$ contains a unique transitive subgroup $G \cong T$ in which the stabiliser of $0_V$ is generated by the matrices $\hA$, $\hB$ in (\ref{hA-hB}). Equivalently, there is a unique surjection $\psi: T \to \F_2^3$ such that (\ref{def-psi}) embeds $T$ as  a transitive subgroup of $\Aff(V)$ and $\psi(A)=\psi(B)=0_V$ for the matrices $A$, $B$ in (\ref{AB}). In total, there are exactly $8$ transitive subgroups of $\Aff(V)$ isomorphic to $T$.
\end{lemma}
\begin{proof} 
Everything except the last sentence follows from the previous discussion. If $G_1 \cong T$ is a transitive subgroup of $\Aff(V)$, then conjugation by a suitable element of $\Aut(V)$ takes the stabiliser of $0_V$ in $G_1$ to $G'=\langle \hA, \hB \rangle$, and hence takes $G_1$ to $G$. Since $G'$ has $8$ conjugates in $\Aut(V)$, there are $8$ possible groups $G_1$.
\end{proof}

\begin{remark}  \label{HGS168}
We interpret Lemma \ref{unique-168} in terms of Hopf-Galois structures. If
$E/K$ is a Galois extension with $\Gal(E/K)=G \cong T$, and $L=E^{G'}$ where $G'$ is a subgroup of $G$ of index $8$, then the existence of a transitive embedding $T \hookrightarrow \Aff(V)=\Hol(V)$ shows that $L/K$ admits at least one Hopf-Galois structure of elementary abelian type $V$. Since there are $8$ transitive subgroups of $\Aff(V)$ isomorphic to $T$, \cite[Proposition 1]{unique} tells us that the number of these Hopf-Galois structures is $8 |\Aut(G,G')|/|\Aut(V)|$, where $\Aut(G,G')$ consists of the automorphisms of $G$ which fix $G'$. Now $|\Out(G)|=2$, the non-trivial element being represented by the automorphism taking each matrix in $\GL_3(2)$ to the transpose of its inverse. Thus $|\Aut(G)|=2 |G|=336$. The automorphisms of $G$ permute the $8$ subgroups of index $8$ transitively, so $|\Aut(G,G')|= |\Aut(G)|/8=42$. Since $\Aut(V)\cong G$, we find that $L/K$ admits $2$ Hopf-Galois structures of type $V$.
\end{remark}

\begin{remark}
Crespo and Salguero \cite{Cr-Sal} have computed all Hopf-Galois structures on  separable field extensions of degree up to $11$. The case in Remark \ref{HGS168} can be found in their paper, although no comment is made there about the appearance of an insoluble Galois group. 
To explain this in more detail, we adopt the notation $8Ti$ used by Crespo and Salguero for the degree $8$ permutation group denoted in \textsc{Magma} by \texttt{TransitiveGroup(8,i)}. Then $T=\GL_3(2)$ is labelled $8T37$, and \cite[Table 2]{Cr-Sal} shows that there are $2$ Hopf-Galois structures for this permutation group, as calculated in Remark \ref{HGS168}. Altogether, there are $5$ insoluble transitive permutation groups of degree $8$: $8T37 \cong \GL_3(2)\cong \PSL_2(7)$, $8T43 \cong \SL_2(7)$, $8T48 \cong \Aff(\F_2^3)$, $8T49 \cong A_8$ and $8T50 \cong S_8$. Of these, only $8T37$ and $8T48$ occur in \cite{Cr-Sal}. For the group $8T48$, the stabiliser of a point is $\PSL_2(7)$, which is not soluble. 
\end{remark}

In Lemma \ref{unique-168}, any non-trivial $G$-admissible subgroup $W$ of $V$ must be 
invariant under the stabiliser $G'=\langle \hA, \hB \rangle$ of $0_V$. Since $G'$ is transitive on $V \bs\{0_V\}$, we have $W=V$. Hence 
Lemma \ref{unique-168} proves Theorem \ref{irred-constr} when $r=1$. We now prove the general case.

\begin{proof}[Proof of Theorem \ref{irred-constr}]

Let $r \geq 1$ and let $H$ be a transitive soluble subgroup of $S_r$. We will construct a transitive subgroup $G \cong T \wr H$ of $\Aff(V)$ where $V=\F_2^{3r}$, and we will
show that the resulting pair of groups $(G,V)$ is a solution to Question \ref{main-qn} such that $V$ contains no non-trivial proper $G$-admissible subgroup. We view elements of $\Aut(V)=\GL_{3r}(2)$ as block matrices, where each block is a $3 \times 3$ matrix over $\F_2$. 
Thus the elements of $\Aff(V) < \GL_{3r+1}(2)$ become block matrices in which each row (except the last)
consists of $r$ blocks followed by a vector in $\F_2^3$. The final row of course consists of $3r$ zeros, followed by $1$.

Let $J \leq \Aff(V)$ be the group of such block matrices of the form
\begin{equation} \label{M-matrix}
 \left( \begin{array}{cccc|c} M_1 & 0 & \cdots & 0 & \psi(M_1) \\ 
                                                   0 & M_2 & \cdots & 0 &  \psi(M_2) \\ 
                                                   \vdots & \vdots &   & \vdots & \vdots \\ 
                                                  0 & 0 & \cdots & M_r & \psi(M_r) \\ \hline
                                                   0 & 0 & \cdots & 0 & 1
                             \end{array} \right),
\end{equation}                            
where $M_1$, $M_2$, \ldots, $M_r \in T=\GL_3(2)$ and $\psi$ is as in Lemma \ref{unique-168}.
Then $J \cong T^r$. We write $J=T_1 \times T_2 \times \cdots \times T_r$, where
$T_i$ consists of the elements of $J$ with $M_k=I$ for all $k \neq i$. We also write $A_i$ (respectively, $B_i$) for the matrix in $T_i$ whose $i$th component is the matrix $A$ (respectively, $B$) of (\ref{AB}), and $T_i'$ for the subgroup generated by $A_i$ and $B_i$. Then the stabiliser of $0_V$ in $J$ is the group $J' = T'_1 \times  T'_2 \times \cdots \times T'_r$. Thus $J'$ is a soluble group of order $21^r$, and the pair $(J,V)$ is a solution of Question \ref{main-qn}. Writing $V=V_1 \oplus \cdots \oplus V_r$ where $V_i$ is the $3$-dimensional subspace corresponding to the $i$th block, we see that each $V_i$ is an $\F_2[J]$-submodule of $V$. Thus $(J,V)$ is not irreducible if $r \geq 2$.

For each permutation $\pi \in S_r$, let $P(\pi) \in \Aff(V)$ be the block matrix whose block in position $(i,j)$ is the identity matrix if $i=\pi(j)$, and the zero matrix otherwise, with the final column of $P(\pi)$ consisting of zero vectors except for a $1$ in the final entry. We have then embedded $S_r$ as a subgroup of $\Aut(V)$ within $\Aff(V)$. Clearly this subgroup normalises $J$. Explicitly, if $X \in J$ is the element in (\ref{M-matrix}) and $\pi \in S_r$, we have 
$$ P(\pi) X P(\pi)^{-1} =                              
\left( \begin{array}{cccc|c} 
      M_{\pi^{-1}(1)} & 0 & \cdots & 0 & \psi(M_{\pi^{-1}(1)}) \\ 
      0 & M_{\pi^{-1}(2)} & \cdots & 0 &  \psi(M_{\pi^{-1}(2)}) \\ 
        \vdots & \vdots &   & \vdots & \vdots \\
        0 & 0 & \cdots &  M_{\pi^{-1}(r)} & \psi(M_{\pi^{-1}(r)}) \\ \hline
                                                   0 & 0 & \cdots & 0 & 1
                             \end{array} \right). $$
We identify the given subgroup $H$ of $S_r$ with its image $P(H)$, and define $G$ to be the subgroup $J \rtimes H \cong T \wr H$ in $\Aff(V)$. 
Then $G$ is a transitive subgroup of $\Aff(V)$ in which the stabiliser of $0_V$ is the soluble group $G'=J' \rtimes H$. Thus $(G,V)$ is a solution of Question \ref{main-qn}. 

Finally, we verify that $V$ contains no non-trivial proper $G$-admissible subspace. In particular, this means that $V$ contains no non-trivial proper $G$-invariant subspace, so the pair $(G,V)$ is irreducible.
Let $W \neq \{0_V\}$ be a $G$-admissible subspace of $V$. Then $W_*$ properly contains $G'$. Thus $W_*$ contains an element $(t_1, \dots, t_r) \in J$ with $t_j \not \in T_j'$ for at least one index $j$. Let $v=(v_1, \ldots, v_r)$ be any element of $V$ satisfying the condition that, for $1 \leq i \leq r$, we have $v_i =0$ if and only if $t_i \in T_i'$. Since $T'$ is transitive on $\F_2^3 \bs \{0_V\}$, there exists $(s_1, \ldots, s_r) \in J'$ so that $(s_1, \ldots, s_r) (t_1, \ldots, t_r) \cdot 0_V = v$. Since $W_*$ is a group containing $J'$, it follows that $v \in W$. Thus $W$ contains the subspace of $V$ generated by all such $v$, which is $\bigoplus_{t_i \not \in T_i'} V_i$. In particular, $V_j \subseteq W$. Since $H \leq W_*$ and $H$ is transitive on the $r$ components, it follows that $W$ contains $V_i$ for every $i$. Hence $W=V$.  
\end{proof}

\section{Generalities on subgroups of affine groups} \label{affine}

The remainder of the paper is devoted to the proof of Theorem \ref{irred-class}.

In this section, we give some preliminary results on the affine group $\Aff(V)$ for $V=\F_p^n$. Propositions \ref{dir-sum} and 
\ref{tensor-product} will be used in Section \ref{conclusion}. 

\subsection{Translations}

\begin{definition}
For any  $G \leq \Aff(V)$, the subgroup of {\em translations} in $G$ is
$$ \Tr(G) = G \cap \{ (v, I) : v \in V\}. $$
\end{definition}

\begin{proposition}  \label{trans}
$\Tr(G)$ is a normal subgroup of $G$ and is isomorphic to the subgroup
$$ U = \{ u \in V : (u,I) \in G \} $$
of $V$. In particular, $\Tr(G)$ is an elementary abelian $p$-group. Moreover, $U$ is a $G$-invariant subspace of $V$, and if $U=V$ then $G/\Tr(G)$ is isomorphic to the stabiliser $G'$ of $0_V$ in $G$.
\end{proposition}
\begin{proof}
$\Tr(G)$ is normal in $G$ since it is the kernel of the composite homomorphism
$G \hookrightarrow \Aff(V) \onto \Aut(V)$. Clearly $U$ is a subspace of $V$, and the map $(v,I) \mapsto v$ gives an isomorphism of groups $\Tr(G) \to U$. 
Moreover, for $u \in U$ and $g=(v,A) \in G$, we calculate
\begin{eqnarray*}
    (v,A)(u,I)(v,A)^{-1} & = & (v+Au,A)(-A^{-1}v, A^{-1})  \\
                                     & = & (v+Au-v,I) \\
                                     & = & (Au,I). 
\end{eqnarray*}                                    
Thus $Au \in U$, so $U$ is a $G$-invariant subspace. Finally, if $U=V$ then for each $g=(v,A) \in G$ we have $(v,I) \in G$ and so $(0_V,A)=(v,I)^{-1}g \in G'$, and the homomorphism $G \to \Aut(V)$ has image $G'$.
\end{proof}

\subsection{Direct sums}

\begin{proposition} \label{dir-sum}
Suppose that $G \leq \Aff(V)$ admits a direct product decomposition $G=A \times B$, and $V$ decomposes as a direct sum
$V=U \oplus W$ of $\F_p[G]$-modules where $U$ is an $\F_p[A]$-module on which $B$ acts trivially, and $W$ is an 
$\F_p[B]$-module on which $A$ acts trivially.
Moreover suppose that the $\F_p[A]$-module $U$ is fixed-point free in the sense that if $u \in U$ and $\theta_\alpha(u)=u$ for all $\alpha=(v_\alpha, \theta_\alpha) \in A$ then $u=0$,
and suppose similarly that $W$ is a fixed-point free $\F_p[B]$-module. Then $A \cdot 0_V \subseteq U$ and $B \cdot 0_V \subseteq W$, so that 
$G \subseteq \Aff(U) \times \Aff(W)$. 
\end{proposition}
\begin{proof}
Choosing bases of $U$ and $W$, we may write elements $\alpha$ of $A$ and $\beta$ of $B$ as block matrices
$$ \alpha = \left( \begin{array}{cc|c} 
      M & 0 & u \\  0 & I & w \\ \hline 0 & 0 & 1  \end{array} \right), \qquad 
       \beta = \left( \begin{array}{cc|c} 
  I & 0 & u' \\  0 & N & w' \\ \hline 0 & 0 & 1  \end{array} \right) $$
where $u$, $u' \in U$ and $w$, $w' \in W$ and the matrices $M$, $N$ correspond to $\theta_\alpha$, $\theta_\beta$. We need to show that $w=0$ and $u'=0$.
Now as $\alpha \beta = \beta \alpha$, we have $Mu'+u  = u+u'$, $w'+w=Nw+w'$. The first of these simplifies to $Mu'=u'$.
Keeping $\beta$ fixed and letting $\alpha \in A$ vary, we have $\theta_\alpha(u')=u'$ for all $\alpha$ in $A$. Thus $u'=0$ since $U$ is fixed-point free.
Similarly, $w=0$.
\end{proof}

\subsection{Tensor products}

\begin{proposition} \label{tensor-product}
Suppose that $G \leq \Aff(V)$ admits a direct product decomposition $G=A \times B$, and that the $\F_p[G]$-module
$V$ decomposes as a tensor product $V=U \otimes_{\F_p} W$ for some $\F_p[A]$-module $U$ and some $\F_p[B]$-module $W$.
Suppose further that $A$ contains an element $\alpha=(0_V, \theta_\alpha)$ fixing $0_V$ such that the endomorphism  
$\theta_\alpha - \id$ of $U$ is invertible. Then $B \cdot 0_V= \{ 0_V\}$.
\end{proposition}
\begin{proof}
Let $(e_j)_{1 \leq j \leq k}$ and $(f_s)_{1 \leq s \leq t}$ be bases for $U$ and $W$.
Then an arbitrary element $\alpha \in A$ acts on the $\F_p[A]$-module $U$ via a 
matrix $M=(m_{ij})$, so that $\alpha(e_j) = Me_j = \sum_i m_{ij} e_i$. Similarly, $\beta \in B$ acts on $W$ via a matrix $N=(n_{rs})$. 
Then $V=U\otimes W$ has basis $(e_j \otimes f_s)_{1 \leq j \leq k, 1 \leq s \leq t}$ and the element $(\alpha,\beta) \in A \times B =G$ acts on the $\F_p[G]$-module $V$  
via the matrix $M \otimes N$: we have $(\alpha,\beta) (e_j \otimes f_s) = (M\otimes N) (e_j \otimes f_s) = \sum_{i,r} m_{ij} n_{rs} e_i \otimes f_r$. 

Now write $\alpha \cdot 0_V = x = \sum_{j, s} x_{js} e_j \otimes f_s$ and $\beta \cdot 0_V = y = \sum_{j, s} y_{js} e_j \otimes f_s$, so that 
$\alpha = (x, M \otimes I)$ and $\beta = (y, I \otimes N)$ in $\Aff(V)=V \rtimes \Aut(V)$.  
Since these two elements must commute, we have
$$   x+(M \otimes I)y = y+(I \otimes N)x. $$
Thus
$$ \sum_{j, s} x_{j,s} e_j \otimes f_s + 
    \sum_{j,s} y_{js} \sum_i m_{ij} e_i \otimes f_s
     = \sum_{j, s} y_{j,s} e_j \otimes f_s + 
         \sum_{j,s} x_{js} \sum_r n_{rs} e_j \otimes f_r.  $$
It follows that, for each pair $j$, $s$, we have
$$ x_{js} + \sum_h m_{jh} y_{hs} = y_{js} + \sum_q  x_{jq} n_{sq}. $$
Rearranging the coefficients of the vectors $x$ and $y$ into matrices 
$X=(x_{js})$ and $Y=(y_{js})$, this gives the matrix equation
$$   (M-I) Y = X(N^t-I), $$
where $N^t$ is the transpose of $N$. If we now choose $\alpha$ as in the statement, so $x=0_V$ and $\theta_\alpha - \id$ is invertible,
then $X=0$ and the matrix $M-I$ is invertible. It follows that $Y=0$. Thus $\beta \cdot 0_V = 0_V$. As $\beta \in B$ was arbitrary, $B \cdot 0_V= \{ 0_V\}$.
\end{proof}

\section{Subgroups with soluble stabilisers of $p$-power index} \label{subgps-affine}

We continue to work in the affine group of $V=\F_p^n$, and now consider subgroups $G$ of $\Aff(V)$ satisfying the following hypothesis:

\begin{hypothesis} \label{p-sol-stab}
For each $v \in V$, the stabiliser of $v$ in $G$ is a soluble subgroup of $p$-power index.
\end{hypothesis}

Hypothesis \ref{p-sol-stab} means that each orbit of $G$ on $V$ has $p$-power cardinality. These orbits need not all have the same cardinality. If $(G,V)$ is a solution to Question \ref{main-qn} then Hypothesis \ref{p-sol-stab} holds since $G$ is transitive on $V$ and each point stabiliser is a soluble subgroup of index $|V|=p^n$.
  
\subsection{A restriction on composition factors}

We will use the following simple observation about permutation groups.

\begin{proposition}  \label{p-normal-orb}
Let $H$ be a group acting faithfully and transitively on a set $X$ of prime-power cardinality $p^r>1$. Let $\{1\} \neq J \lhd H$. Then there exist $s \geq 0$ and $t \geq 1$ such that $J$ has $p^s$ orbits on $X$, each of these orbits has cardinality $p^t$, and $H/J$ permutes these orbits transitively.
\end{proposition}
\begin{proof}
Let $H_x$ be the stabiliser in $H$ of $x \in X$. For any $h \in H$, the stabiliser of $h \cdot x$ in $J$ is $h H_x h^{-1} \cap J = h(H_x \cap J)h^{-1}$.  The order of this stabiliser is independent of $h$, so the $J$-orbits on $X
$ all have the same cardinality. Since these orbits form a partition of $X$, the number of orbits divides $|X|$. Thus, for some $s \geq 0$, there are $p^s$ orbits, and each has cardinality $p^t$ where $t=r-s$. As $|J|>1$ and $J$ acts faithfully on $X$, we have $t \geq 1$. Also $H/J$ has a well-defined action on 
the $p^s$ orbits $J \cdot x$ since $h \cdot (J \cdot x) = J \cdot (h \cdot x)$ as $J \lhd H$. This action is transitive since $H$ is transitive on $X$. 
\end{proof}

\begin{lemma} \label{comp-factors}
Let $(G,V)$ satisfy Hypothesis \ref{p-sol-stab}, and let 
\begin{equation}  \label{comp-series}
   \{1\} =G_0 \triangleleft G_1 \triangleleft \cdots \triangleleft  G_m=G 
\end{equation}
be any composition series for $G$. Then the following hold for $1 \leq i \leq m$:
\begin{itemize}
\item[(i)] the pair $(G_i,V)$ also satisfies Hypothesis \ref{p-sol-stab};
\item[(ii)] the simple group $G_i/G_{i-1}$ has a soluble subgroup of index $p^s$ for some $s \geq 0$.
\end{itemize}
\end{lemma}
\begin{proof}
We show (i) by descending induction on $i$. If $i=m$ then Hypothesis \ref{p-sol-stab} holds by assumption.  Suppose that $i>1$ and $(G_i,V)$ satisfies Hypothesis \ref{p-sol-stab}, and let $v \in V$. Then the stabiliser $G'_i(v)$ of $v$ in $G_i$ is soluble of $p$-power index, so the orbit $G_i \cdot v$ has $p$-power cardinality. Then $G_{i-1} \cdot v$ also has $p$-power cardinality by 
Proposition \ref{p-normal-orb} with $X=G_i \cdot v$, $H=G_i$, $J=G_{i-1}$. Moreover the stabiliser $G'_{i-1}(v)$ of $v$ in $G_{i-1}$ is soluble, being a subgroup of $G'_i(v)$. Hence $(G_{i-1},V)$ also satisfies Hypothesis \ref{p-sol-stab}. This completes the induction.
	
For (ii), we apply Proposition \ref{p-normal-orb} with $X=G_i \cdot 0_V$, $H=G_i$, $J=G_{i-1}$. Let $G'$ be the stabiliser of $0_V$ in $G$. For some $s \geq 0$, the set $X$ decomposes into $p^s$ orbits under $G_{i-1}$, and $G_i/G_{i-1}$ acts transitively on these orbits. The stabiliser of the orbit $G_{i-1} \cdot 0_V$ in this action is $G_{i-1} G'_i/G_{i-1}$, where $G'_i=G' \cap G_i$ is the stabiliser of $0_V$ in $G_i$. Since $G'$ is soluble by hypothesis, the groups $G'_i$ and 
$G_{i-1}G'_i/G_{i-1} \cong G'_i/(G'_i \cap G_{i-1})$ are also soluble. Thus $G_i/G_{i-1}$ contains the soluble subgroup $G_{i-1}G'_i/G_{i-1}$ of index $p^s$.
\end{proof}

\subsection{Subnormal subgroups in the insoluble irreducible case}

\begin{lemma} \label{nilp}
Let $H$ be a finite group, and let $P$ be a normal $p$-subgroup of $H$. Then the kernel $\NN$ of the canonical ring homomorphism $\F_p[H] \to \F_p[H/P]$ is a nilpotent ideal of $\F_p[H]$.
\end{lemma}
\begin{proof}
Let $F$ be the Frattini subgroup of $P$. Then $F \lhd H$, and $P/F$ is an elementary abelian $p$-group. Moreover, $F \lneq P$ unless $P$ is trivial.

Let $\MM$ be the kernel of the canonical homomorphism $\F_p[H] \to \F_p[H/F]$. It suffices by induction to show that $\NN^k \subseteq \MM$ for some $k>0$. Replacing $H$ and $P$ by $H/F$ and $P/F$, we may therefore assume that $P$ is elementary abelian.

Now $\NN$ is the subspace of $\F_p[H]$ spanned by all elements of the form $h(\pi-1)$ for $h \in H$ and $\pi \in P$; this subspace is visibly a left ideal, and it is a right ideal since
\begin{equation} \label{rt-ideal}   
 h_1(\pi-1)h_2 = h_1 h_2 (\pi'-1) \mbox{ with } \pi'=h_2^{-1} \pi h_2 \in P. 
\end{equation}
For $k>0$, we can use (\ref{rt-ideal}) repeatedly to write the product of any $k$
elements of the form $h(\pi-1)$ as
\begin{equation} 
	h^* (\pi_1-1) (\pi_2-1) \cdots (\pi_k-1)
\end{equation}
with $h^* \in H$ and $\pi_1, \ldots, \pi_k \in P$. Moreover, since $P$ is abelian, we may reorder the factors $\pi_i-1$. Taking $k=(p-1)|P|+1$, the product can therefore be rewritten to contain $(\pi-1)^p$ for some $\pi \in P$. Since 
$$ (\pi-1)^p=\pi^p-1=0,   $$ 
this shows that $\NN^k=\{0\}$.
\end{proof}

\begin{lemma} \label{p-normal}
Suppose that $G$ is insoluble and satisfies Hypothesis \ref{p-sol-stab}, and that the pair $(G,V)$ is irreducible. Then $G$ has no non-trivial normal $p$-subgroup.
\end{lemma}
\begin{proof}
Let $P$ be a non-trivial normal $p$-subgroup of $G$, let $\JJ$ be the Jacobson radical of $\F_p[G]$, and let $\NN$ be the kernel of the canonical ring homomorphism $\F_p[G] \to \F_p[G/P]$.
Then $\NN \subseteq \JJ$ since $\NN$ is nilpotent by Lemma \ref{nilp}. Let $W$ be the subspace $\NN V$ of $V$. Then $W$ is a $G$-invariant subspace since $\NN$ is an ideal. 
Now if $W=V$ we have
$$   V = W = \NN V \subseteq \JJ V,  $$
so by Nakayama's Lemma (applied to the finitely generated $\F_p[G]$-module $V$), we have $V=\{0_V\}$, contradicting the hypotheses. Since $V$ is irreducible, we therefore have $W=\{0_V\} \neq V$.

For each $\pi=(u_\pi,A_\pi) \in P$ we have $\pi -1 \in \NN$, so $(A_\pi-I)v=0_V$ for all $v \in V$. Hence $P$ consists of translations. Then $\Tr(G)$ is non-trivial, and by Proposition \ref{trans}, $U = \{ u \in V: (u,I) \in G \}$ is a non-trivial $G$-invariant subspace of $V$. Using the irreducibility of $V$ again, it follows that $U=V$. 
By Proposition \ref{trans} again, $G/\Tr(G)$ is isomorphic to the stabiliser $G'$ of $0_V$. This is a contradiction since $G$ is insoluble but $G'$ and $\Tr(G)$ are soluble. 
\end{proof} 

\begin{corollary}  \label{no-sol-subnormal}
If $(G,V)$ is as in Lemma \ref{p-normal}, then $G$ has no non-trivial soluble subnormal subgroup.
\end{corollary}
\begin{proof}
Let $H$ be such a subgroup. Then there is a composition series (\ref{comp-series}) of $G$ with $G_i=H$ for some $i \geq 1$.  
By Lemma  \ref{comp-factors}(i), the group $G_1$ in this series acts faithfully on $V$ with orbits of $p$-power cardinality. Since $G_1$ is a soluble simple group, we have $|G_1|=p$.  By \cite[9.1.1]{Robinson}, $G$ contains a unique maximal normal $p$-subgroup $O_p(G)$, and $G_1 \subseteq O_p(G)$. Thus $O_p(G) \neq \{1\}$, contradicting Lemma \ref{p-normal}.
\end{proof}

\subsection{The socle of $G$}

Recall that for a finite group $H$, the socle $\soc(H)$ of $H$ is the subgroup generated by all the minimal normal subgroups. Then $\soc(G)$ can be written as the direct product of some of these minimal normal subgroups \cite[p.~87]{Robinson}. The next result describes a situation when all the minimal normal subgroups are needed.

\begin{proposition}  \label{socle}
Let $H$ be a finite group such that every minimal normal subgroup of $H$ has trivial centre. Then $\soc(H)$ is the direct product of all the minimal normal subgroups of $H$.
\end{proposition}
\begin{proof}
Let $J_1$, \ldots, $J_s$ be the distinct minimal normal subgroups of $H$. 
We will show by induction that, for $1 \leq i \leq s$, the subgroup $K_i$ of $H$ generated by $J_1$, \ldots, $J_i$ is the direct product $J_1 \times \cdots \times J_i$. This is clear for $i=1$. Assume that $i>1$ and that $K_{i-1}=J_1 \times \cdots \times J_{i-1}$. We claim that $J_i \cap K_{i-1} = \{1\}$. If $1 \neq a \in J_i \cap K_{i-1}$ we may write $a=(a_1, \ldots, a_{i-1})$ with each $a_h \in J_h$, and, without loss of generality, $a_1 \neq 1$. Taking the commutator with $(b,1, \ldots, 1)$ for arbitrary $b \in J_1$, we have $(b^{-1}a_1^{-1}ba_1,1, \ldots, 1) \in J_i \cap K_{i-1}$. Thus $b^{-1}a_1^{-1}ba_1 \in J_1 \cap J_i$ for all $b \in J_1$. As $a_1$ is not in the centre of $J_1$, we have $J_1 \cap J_i \neq \{1\}$. This is impossible since $J_1$ and $J_i$ are distinct minimal normal subgroups. Hence 
$J_i \cap K_{i-1} = \{1\}$ as claimed, and in particular $J_i$ centralises $K_{i-1}$. It follows that the subgroup of $H$ generated by the normal subgroups $K_{i-1}$ and $J_i$ is their direct product, that is, $K_i=J_1 \times \cdots \times J_i$. 
\end{proof}

\begin{proposition} \label{d-prod-simple}
Let $H=J_1 \times \cdots J_r$ be a finite group which is the direct product of nonabelian simple subgroups $J_i$, $1 \leq i \leq r$. Then the normal subgroups of $H$ are precisely the direct products of subsets of $\{J_1, \ldots, J_r\}$. 
\end{proposition}
\begin{proof}	
Clearly the direct product of a subset of $\{J_1, \ldots, J_r\}$ is a normal subgroup of $H$. For the converse, let $K \lhd H$ and let $K_i$ be the projection of $K$ to the direct factor $J_i$ in $H$. If suffices to show that, for each $i$, either $K_i=\{1\}$ or $K$ contains $J_i$. 
Suppose that $K_i \neq \{1\}$. Since $K_i \lhd J_i$ and $J$ is simple, we have $K_i = J_i$. Let $a$, $b \in J_i$ be arbitrary. As $a \in K_i$, there is some $\alpha=(a_1, \ldots, a_r)\in K$ with $a_i=a$. Let $\beta=(1, \ldots, b, 1, \ldots) \in K$, with $b$ in the $i$th factor. Then
$$ \beta^{-1} \alpha^{-1} \beta \alpha = ( 1, \ldots, b^{-1} a^{-1} b a, 1, \ldots) \in K.  $$
Since $a$, $b$ are arbitrary elements of $J_i$, it follows that the derived subgroup $[J_i,J_i]$ of $J_i$ is contained in $K$. But $[J_i,J_i]=J_i$ since $J_i$ is a non-abelian simple group. Hence $J_i \subseteq K$ as required.
\end{proof}

\begin{lemma}  \label{irred-soc}
Let $G$ be an insoluble subgroup of $\Aff(V)$ satisfying Hypothesis \ref{p-sol-stab}, and suppose that $(G,V)$ is irreducible. Let $S=\soc(G)$.
Then $S$ is a direct product of non-abelian simple groups:
\begin{equation} \label{soc-G} 
  S = T_1 \times \cdots \times T_r. 
\end{equation}  
Conjugation by $G$ permutes $\{T_1, \ldots, T_r\}$, and each minimal normal subgroup of $G$ is the direct product of all the $T_i$ lying in a single orbit.
\end{lemma}
\begin{proof}
Let $J$ be a minimal normal subgroup of $G$. Then $J \lhd K$ for some minimal characteristic subgroup $K$ of $G$. As $K$ is characteristically simple, it  is the direct product of isomorphic simple groups \cite[3.3.15]{Robinson}, say $K=T_1 \times \cdots \times T_m$.
By Corollary \ref{no-sol-subnormal}, these simple groups are non-abelian. 
By Proposition \ref{d-prod-simple}, a normal subgroup of $K$ is just the direct product of some of the groups $T_k$, $1 \leq k \leq m$. In particular, $J$ has trivial centre. Thus, by Proposition \ref{socle},  
$S$ is the direct product of all the minimal normal subgroups of $G$, each of which is itself a direct product of non-abelian simple groups. Hence $S$ is a direct product of non-abelian simple groups as in (\ref{soc-G}). 
  
The minimal normal subgroups of $S$ are the groups $T_k$ for $1 \leq k \leq r$. For any $g \in G$, the group $gT_k g^{-1}$ is also a minimal normal subgroup of $S$, so coincides with some $T_j$. Hence conjugation by $G$ permutes the $T_k$. Any minimal normal subgroup $J$ of $G$ is also a normal subgroup of $S$, so, by Proposition \ref{d-prod-simple}, it must be the direct product of some of the $T_k$. If it contains $T_k$ it must also contain $g T_k g^{-1}$ for all $g \in G$. Thus $J$ contains the direct product of all the conjugates of $T_k$. As this direct product is itself a normal subgroup of $G$, it coincides with $J$. Thus the minimal normal subgroups of $G$ correspond to the $G$-orbits on $\{T_1, \ldots, T_r\}$ as described. 
\end{proof}

\begin{corollary}  \label{centraliser}
Let $G$ be as in Lemma \ref{irred-soc}. Then $G$ embeds in $\Aut(\soc(G))$.
\end{corollary}
\begin{proof}
Again let $S = \soc(G)$. As $S \lhd G$, conjugation gives a homomorphism $G \to \Aut(S)$ whose kernel is the centraliser $C$ of $S$ in $G$. Thus it suffices to show that $C$ is trivial. Now $C$ is a characteristic subgroup of $G$. If $C \neq \{1\}$ then $C$ contains a minimal normal subgroup $D$ of $G$. But then $D \leq S$, so $D$ is contained in the centre of $S$. This is a contradiction since $S$ has trivial centre by Lemma \ref{irred-soc}.
\end{proof}

\section{Clifford theory}  \label{clifford}

In this section, $(G,V)$ is again an irreducible solution to Question \ref{main-qn} with $V=\F_p^n$. By Proposition \ref{trans} and Lemma \ref{p-normal}, the subgroup of translations in $G$ is trivial, so $G$ acts faithfully on the irreducible $\F_p[G]$-module $V$.

Let $S=\soc(G)$. 
It follows from Lemma \ref{irred-soc} and Proposition \ref{socle} that $S$ is the direct product of the minimal normal subgroups of $G$, each of which is in turn the direct product of isomorphic non-abelian simple groups permuted transitively by $G$. We now consider the representation of the normal subgroup $S$ of $G$ afforded by $V$. In other words, we apply Clifford theory to our situation. 

It is convenient to extend scalars to a finite extension $E$ of $\F_p$ which splits $G$. Then $E$ also splits every subgroup $H$ of $G$, so every irreducible $E[H]$-module is absolutely irreducible.
We write $V_E=E \otimes_{\F_p} V$.

\begin{proposition}  \label{VE} 
$V_E$ is a direct sum of irreducible $E[G]$-modules and $G$ acts faithfully on each of these submodules.
\end{proposition}
\begin{proof}
By \cite[(7.11)]{CR}, $V_E$ is a direct sum of irreducible $E[G]$-modules. Moreover,
these modules are all conjugate under the action of $\Omega = \Gal(E/\F_p)$. This follows from the proof of \cite[(7.11)]{CR}) since, if $k$ is a subfield of $E$ of degree $d$ over $\F_p$, then we have an isomorphism $E \otimes_{\F_p} k \to E^d$ of $E$-algebras given by $x \otimes y \mapsto ( x \sigma(y))_{\sigma \in \Gal(k/\F_p)}$.

Let $W$ be an irreducible $E[G]$-submodule of $V_E$. If $g \in G$ acts trivially on $W$, then it also acts trivially on all Galois conjugates of $W$. Thus $g$ acts trivially on $V_E$. Since $G$ acts faithfully on $V$, this forces $g=1$. Hence $G$ acts faithfully on $W$.
\end{proof}

We will also need the following fact.

\begin{proposition} \label{irred-tensor}
Let $H$ be any subgroup of $G$ which decomposes  as a direct product $H=A \times B$. Then any irreducible $E[H]$-module has the form $X \otimes_E Y$ where $X$ (respectively, $Y$) is an irreducible $E[A]$-module (respectively, $E[B]$-module).
\end{proposition}
\begin{proof}
As $E$ splits $A$ and $B$, the semisimple algebras $E[A]/\rad(E[A])$ and $E[B]/\rad(E[B])$
are split. The result now follows from \cite[(10.38)(iii)]{CR}.
\end{proof}

Let $W$ be any of the irreducible $E[G]$-modules in Proposition \ref{VE}. Clearly we have
\begin{equation} \label{VE-dim}
 \dim_E W \leq \dim_E V_E = \dim_{\F_p} V  . 
 \end{equation}
Next let $U$ be an irreducible $E[S]$-submodule of $W$.
For each $g \in G$, the subspace $gU$ of $W$ is an irreducible $E[S]$-module since $S \lhd G$. Then $W$ is the sum of these modules. It follows that $W$ is a semisimple $E[S]$-module, and hence $W$ is the direct sum of some of the $gU$:
\begin{equation} \label{U-sum}
  W = \bigoplus_{i=1}^m g_i U, 
\end{equation}
with $g_1, \ldots, g_m \in G$. Without loss of generality, we take $g_1=1$.

Now let $J$ be a minimal normal subgroup of $G$, so by Lemma \ref{irred-soc} 
\begin{equation} \label{J-prod}
   J = T_1 \times \cdots \times T_{r(J)}, 
\end{equation}
for some $r(J) \geq 1$, where conjugation by $G$ permutes $T_1, \ldots, T_{r(J)}$ transitively. Thus there is a non-abelian simple group $T_J$ such that $T_k \cong T_J$ for $1 \leq k \leq r(J)$.  Each summand  $g_i U$ in (\ref{U-sum}) is an $E[T_k]$-module for each of the groups $T_k$. Choosing $i=1$, we define $y(J)$ to be the number of groups $T_k$ which act non-trivially on $U=g_1 U$. Similarly, choosing $k=1$, we define $z(J)$ to be the number of summands $g_i U$ on which $T_1$ acts non-trivially. (We shall see that $y(J)$ and $z(J)$ are in fact independent of these choices.)

\begin{proposition} \label{mk-ry-prop}
For each minimal normal subgroup $J$ of $G$, we have 
\begin{equation} \label{mk-ry}
     m y(J) = r(J) z(J).
\end{equation}
Moreover, $y(J) \geq 1$ and $z(J) \geq 1$.
\end{proposition}
\begin{proof}
We count in two ways the number of pairs $(i,k)$ such that $T_k$ acts non-trivially on $g_i U$.  Firstly, for each $i$, the group
$T_k$ acts non-trivially on $g_i U$ if and only if $g_i^{-1} T_k g_i$ acts non-trivially on $U$. Since conjugation by $g_i^{-1}$ permutes the groups $T_k$, and $y(J)$ of the groups $T_k$ act non-trivially on $U$, it follows that 
there are $y(J)$ groups $T_k$ which act non-trivially on $g_i U$. (This shows that $y(J)$ does not depend on the initial choice $i=1$.) Thus the number of pairs $(i,k)$ as above is $my(J)$. 
Secondly, for $1 \leq k \leq r(J)$ we have $T_k=gT_1g^{-1}$ for some $g \in G$, so that $T_k$ acts non-trivially on $g_i U$ if and only if $T_1$ acts non-trivially on $g^{-1} g_i U$. There are $z(J)$ values of $i$ for which $T_1$ acts non-trivially on $g^{-1} g_i U$ since the two decompositions of $W$ into irreducible $E[S]$-modules,
$$   W = \bigoplus_{i=1}^m g_i U = \bigoplus_{i=1}^m g^{-1}\ g_i U, $$
must contain isomorphic summands. Thus there are $z(J)$ values of $i$ for which $T_k$ acts non-trivially on $g_i U$. (This shows that $z(J)$ does not depend on the initial choice $k=1$.)
There are therefore $r(J)z(J)$ pairs $(i,k)$ such that $T_k$ acts non-trivially on $g_i U$. Hence (\ref{mk-ry}) holds. Moreover, as the non-abelian group $T_1$ cannot be contained in $\Tr(G)$, its linear action on $W$ is non-trivial. Thus $z(J) \geq 1$. Hence by (\ref{mk-ry}) we also have $y(J) \geq 1$. 
\end{proof}

For a non-abelian simple group $T$, let $d_p(T)$ denote the minimal degree of a non-trivial irreducible $F[T]$-module, where $F$ is any splitting field for $T$ in characteristic $p$. Then $d_p(T) \geq 2$.

In the following, we write $v_p(x)$ for the $p$-adic valuation of an integer $x \neq 0$: $v_p(x)=a$ if $p^a$ divides $x$ but $p^{a+1}$ does not. 

\begin{lemma} \label{key-ineq-lemma}
With the above notation,
\begin{equation} \label{key-ineq}
  m \prod_J d_p(T_J)^{y(J)} < 
        \sum_J r(J) \left( v_p(|\Aut(T_J)|)  + \frac{1}{p-1}\right), 
\end{equation}
the product and sum being over all minimal normal subgroups $J$ of $G$.
\end{lemma}
\begin{proof}
Using Proposition \ref{irred-tensor} repeatedly, we find that the irreducible $E[S]$-module $U$ decomposes as a tensor product over $E$:
\begin{equation} \label{U-tensor}
    U = \bigotimes_J U_J,
\end{equation}
where each factor $U_J$ is an irreducible $E[J]$-module.
Moreover, from Proposition \ref{irred-tensor} again and (\ref{J-prod}), we have 
$$	 U_J = U_{J,1} \otimes_E \cdots \otimes_E U_{J,r(J)} $$
where $U_{J,k}$ is an irreducible $E[T_k]$-module. There are precisely $y(J)$ values of $k$ such that $U_{J,k}$ is a non-trivial $E[T_k]$-module. (For the remaining values of $k$, we have that $U_{J,k}$ is the $1$-dimensional trivial $E[T_k]$-module.) Hence we have 
$\dim_E(U_J)  \geq d_p(T_J)^{y(J)}$ for each minimal normal subgroup $J$. It then follows from (\ref{U-sum}) and (\ref{U-tensor}) that  
\begin{equation} \label{dimW}
   \dim_E W =m \dim_E U \geq m \prod_J d_p(T_J)^{y(J)}.
\end{equation}

Let $r= \sum_J r(J)$. Then $S$ is the direct product of $r$ non-abelian simple groups:
$$ S = \prod_J J \cong \prod_J  T_J^{r(J)}. $$
Since any automorphism of $S$ must permute the $r$ simple groups in this product, it follows that $\prod_J  \Aut(T_J)^{r(J)}$ is a normal subgroup of $\Aut(S)$ with quotient isomorphic some subgroup of $S_r$. (This subgroup cannot be transitive if $G$ has more than one minimal normal subgroup.) Now $G$ embeds in $\Aut(S)$ by Corollary \ref{centraliser}, so
\begin{eqnarray*}
   v_p(|G|) & \leq & v_p(|\Aut(S)|) \\
     & \leq & v_p \left( \left| \prod_J  \Aut(T_J) ^{r(J)}\right| \right) 
                     + v_p \left( |S_r| \right)  \\
         & < & \left(  \sum_J r(J)  v_p(|\Aut(T_J)|)\right)  + \frac{r}{p-1} \\   
                                 & = & \sum_J r(J) \left( v_p(|\Aut(T_J)|)  + \frac{1}{p-1}\right).
 \end{eqnarray*} 
On the other hand, since $G$ acts transitively on $V$, we have
\begin{equation} \label{dim-V-G}
   v_p(|V|)=\dim_{\F_p} V \leq v_p(|G|). 
\end{equation}
Using (\ref{VE-dim}) and (\ref{dimW}), we then have (\ref{key-ineq}).
\end{proof}

We summarise our conclusions so far.

\begin{theorem} \label{no-CFSG}
Let $(G,V)$ be an irreducible solution to Question \ref{main-qn}: $V$ is a finite soluble 
group, $G$ is an insoluble 
transitive subgroup of $\Hol(V)$ such that $V$ has no non-trivial proper $G$-invariant normal subgroup, and the stabiliser in $G$ of any element of $V$ is soluble. Then
\begin{itemize}
\item[(i)] $V$ is an elementary abelian group: $V \cong \F_p^n$ for some prime $p$ and some $n \geq 1$, and,
viewed as a subgroup of $\Aff(V)$, the group $G$ contains no non-trivial translations;
\item[(ii)] $\soc(G)$ is the direct product of all the minimal normal subgroups $J$ of $G$,  
each of which is in turn a direct product $J \cong T_J^{r(J)}$ for some non-abelian simple group $T_J$ and some $r(J) \geq 1$;
\item[(iii)] $G$ embeds in $\Aut(\soc(G))$;
\item[(iv)] each of the non-abelian simple groups $T_J$ has a soluble subgroup of $p$-power index;
\item[(v)] there is an integer $m$, and integers $y(J)$, $z(J)$ for each $J$, such that $1 \leq y(J) \leq r(J)$, $1 \leq z(J) \leq m$, and 
(\ref{mk-ry}) and (\ref{key-ineq}) hold.
\end{itemize}
\end{theorem}
\begin{proof}
(i) follows from Lemma \ref{irred-el-ab}, Proposition \ref{trans} and Lemma \ref{p-normal}; 
(ii) from Lemma \ref{irred-soc}; (iii) from Corollary \ref{centraliser}; (iv) from 
Lemma \ref{comp-factors}(ii); and (v) recapitulates Proposition \ref{mk-ry-prop} and Lemma \ref{key-ineq-lemma}.
\end{proof}

Theorem \ref{no-CFSG} gives very stringent conditions on $(G,V)$ since, 
on the one hand, it is rare for a non-abelian simple group to have a soluble subgroup of prime-power index and, on the other hand, most choices of the 
numerical parameters will make the product on the left of (\ref{key-ineq}) exceed the sum on the right. We know from Theorem \ref{irred-constr}, however, that these conditions can be satisfied.

\section{Applying the classification of finite simple groups} \label{cfsg}

As a consequence of the Classification of Finite Simple Groups, Guralnick \cite{Gural} proved the following result:

\begin{theorem}[Guralnick] \label{Guralnick}
If $T$ is a non-abelian finite simple group with a proper subgroup $R$ of
prime-power index $p^a$, then one of the following holds. 
\begin{itemize} 
\item[(i)] $T=A_n$ and $R=A_{n-1}$ with $n=p^a$;
\item[(ii)] $T=\PSL_n(q)$, $p^a=(q^n-1)/(q-1)$ and $R$ is the
  stabiliser of a line or a hyperplane in $\F_q^n$; 
\item[(iii)] $T=\PSL_2(11)$ and $R=A_5$ of index $11$;
\item[(iv)] $T=M_{23}$, $R=M_{22}$ or $T=M_{11}$, $R=M_{10}$;
\item[(v)] $T=\PSU_4(2) \cong \PSp_4(3)$ and $R$ has index $27$.
\end{itemize}
\end{theorem}

We deduce from this a list of candidates for the simple groups $T_J$ and primes $p$ in Theorem \ref{no-CFSG}.

\begin{corollary}  \label{simple-sol-p}
If $T$ is a non-abelian finite simple group with a soluble subgroup
of prime-power index $p^a$, then the triple $(T,p,a)$ is one of the following: 
\begin{itemize}
	\item[(i)] $(\PSL_3(2),7,1)$; 
\item[(ii)] $(\PSL_3(3), 13, 1)$;
\item[(iii)] $(\PSL_2(2^s), p,1)$ where $p=2^s+1$ is a Fermat prime with $p \geq 5$;
\item[(iv)] $(\PSL_2(8),3,2)$;
\item[(v)] $(\PSL_2(q),2,a)$ where $q=2^a-1$ is a Mersenne prime with $q \geq 7$. 
\end{itemize}
\end{corollary}

\begin{proof}
We check the cases of Theorem \ref{Guralnick}. The only instance of
Theorem \ref{Guralnick}(i) with $R$ soluble is $n=5$, and the exceptional isomorphism $A_5 \cong \PSL_2(4)$ means that this is included in case (iii) of the Corollary. In cases (iii) and
(iv) of Theorem \ref{Guralnick}, $R$ is insoluble. The same holds in
case (v) since  $R$ has a quotient
isomorphic to $A_5$ (see the proof of \cite[(3.9)]{Gural}).

This leaves only case (ii) of Theorem \ref{Guralnick}. We first consider $n \geq 3$. Then the stabiliser of a line or a hyperplane contains $\PSL_{n-1}(q)$ as a subquotient, and this is a non-abelian simple group unless $n=3$ and $q=2$ or $3$. If $n=3$ and $q=2$, 
we have $T=\PSL_3(2)$ with a subgroup of index $7$, giving (i) of the Corollary. If $n=3$ and $q=3$, we have $T=\PSL_3(3)$ with a subgroup of index $13$, giving (ii).

Now let $n=2$, so $T=\PSL_2(q)$ where $q = \ell^s \geq 4$ is a
power of some prime $\ell$, and $q+1=p^a$ is also a prime power. Either $\ell=2$ or
$p=2$. 

If $\ell=2$, we have $q=2^s=p^a-1$ for an odd prime $p$ and some $a \geq 1$. If $a=1$ then $p$ is a Fermat prime and we have (iii). The case $p=3$ is
omitted as $\PSL_2(2) \cong S_3$ is not a simple group. If $a>1$ then $(p-1)(p^{a-1}+ \cdots + 1)=p^a-1=2^s$ so that 
$p-1=2^u$ and $p^{a-1}+\cdots + 1 =2^v$ with $1 \leq u <v$. Then $p \equiv 1 \pmod{2^u}$ and so $a \equiv p^{a-1}+ \cdots +1 \equiv 0 \pmod{2^u}$.  Thus $a$ is even, say $a=2b$.
But then $(p^b-1)(p^b+1)=2^s$. As both factors must be powers of $2$, we have $p^b=3$. This gives (iv).

Finally, if $p=2$ then $\ell^s=2^a-1$ for some $a \geq 2$. 
Thus $\ell^s \equiv -1 \pmod{4}$, so $s$ and $\ell$ are odd. If $s>1$ we have 
$$ (\ell+1)(\ell^{s-1}-\ell^{s-2}+ \cdots +1) = 2^a. $$
Then $2^a$ has an odd factor $\ell^{s-1}-\ell^{s-2}+ \cdots - \ell +1>1$, giving a contradiction. Hence $s=1$ and $q=\ell$ is prime with $q=2^a-1$. This
says that $q$ is a Mersenne prime, giving (v). We have $q \geq 7$ since $q=3$ gives $\PSL_2(3) \cong A_4$ which is not a simple group.
\end{proof}

In fact, only the simple group of order $168$ is relevant.

\begin{theorem} \label{just-168}
Let $(G,V)$ be an irreducible solution to Question \ref{main-qn}. Then, in the notation of Theorem \ref{no-CFSG}, $p=2$ and for each minimal normal subgroup $J$ of $G$, the simple group 
$T_J$ is the group $T=\PSL_2(7) \cong \GL_3(2)$ of order $168$. Moreover, one of the following holds:
\begin{itemize}
\item[(i)] $G$ contains exactly two minimal normal subgroups $J_1\cong J_2 \cong T^r$ with $r=m$, $y(J_1)=y(J_2)=z(J_1)=z(J_2)=1$;
\item[(ii)] $G$ contains a unique minimal normal subgroup $J=\soc(G) \cong T^r$ with $r=2m$, $y(J)=2$, $z(J)=1$;
\item[(iii)] $G$ contains a unique minimal normal subgroup $J=\soc(G) \cong T^r$ with $r=m$, $y(J)=z(J)=1$.
\end{itemize}
\end{theorem}
\begin{proof}
By Theorem \ref{no-CFSG}(iv), each $T_J$ has a soluble subgroup of index $p^a$ for some $a \geq 1$. The triple $(T_J,p,a)$ must therefore be one of those listed in Corollary \ref{simple-sol-p}. 
In the sequel, we will use the fact that, for $q$ prime and $s \geq 1$,
$$  |\Out(\PSL_n(q^s))| = \begin{cases} s\,\gcd(2,q^s-1) & \mbox{if } n=2, \\
		                                                   2s \, \gcd(2,q^s-1) & \mbox{if } n \geq 3. \end{cases} $$

We first consider the cases where $p>2$. 
Since $d_p(T_J)^{y(J)} \geq 2$ for each $J$, we may replace the product in (\ref{key-ineq}) by a sum. 
Using (\ref{mk-ry}), we then have
$$   \sum_J r(J) \left( \frac{z(J)}{y(J)} d_p(T_J)^{y(J)}
                         - v_p(|\Aut(T_J)|) - \frac{1}{p-1} \right)<0. $$                         
As $z(J) \geq 1$, it follows that there must be at least one $J$ for which
\begin{equation} \label{TJ-ineq}
 \frac{1}{y(J)} d_p(T_J)^{y(J)} <  v_p(|\Aut(T_J)|) + \frac{1}{p-1}. 
\end{equation}
In cases (i)--(iii) of Corollary \ref{simple-sol-p} we have $v_p(|\Out(T_J)|)=0$ and \\ 
$v_p(|\Aut(T_J)|)=v_p(|T_J|)=1$.
As $d_p(T_J) \geq 2$, (\ref{TJ-ineq}) cannot be satisfied for any integer $y(J) \geq 1$. In case (iv), $p=3$ and we have 
$v_3(|T_J|)=2$ and $v_3(|\Out(T_J)|)=1$, so that $v_3(|\Aut(T_J)|)=3$, but from \cite{online-atlas} we have $d_3(T_J)=7$ and again (\ref{TJ-ineq}) cannot hold.
Thus at least one of the $T_J$ must be as in Corollary \ref{simple-sol-p}(v). Hence $p=2$, and, for each $J$, the group $T=T_J$ has the form $T=\PSL_2(q)$ for some Mersenne prime $q=2^a-1$ with 
$a \geq 3$. Then $|T|=q(q+1)(q-1)/2$, so 
$v_2(|T|)=a$ and $v_2(|\Aut(T)|)=a+1$. From \cite[\S VIII]{Burkhardt} we have $d_2(T)=(q-1)/2 = 2^{a-1}-1$. Writing $T_J=\PSL_2(q(J))$ with $q(J) =2^{a(J)}-1$, we find from (\ref{TJ-ineq}) that
$$   \frac{(2^{a(J)-1}-1)^{y(J)}}{y(J)}  < a(J)+2.   $$
This holds only for $a(J)=3$ and $y(J)=1$ or $2$. 

We have now shown that, for at least one of the minimal normal subgroups $J$ of $G$, we have $a(J)=3$ and thus $T_J \cong \PSL_2(7) \cong \GL_3(2)$. We next show that in fact $T_J \cong \PSL_2(7)$ for all $J$. To do so, we return to (\ref{key-ineq}), which now becomes
$$ m \prod_J d_2(T_J)^{y(J)} <
        \sum_J r(J) (a(J)+2). $$
Dividing by $m$ and using (\ref{mk-ry}), we then have
\begin{equation} \label{key-ineq-bis}
   \prod_J d_2(T_J)^{y(J)} <
        \sum_J \frac{y(J) (a(J)+2)}{z(J)}. 
\end{equation}
Since each $z(J) \geq 1$, it follows that 
$$   \prod_J d_2(T_J)^{y(J)} < \sum_J y(J) (a(J)+2).  $$
To proceed further, we collect together the minimal normal subgroups $J$ with isomorphic simple components $T_J$.
Let $a_1=3, \ldots, a_h$ be the distinct values occurring among the $a(J)$. For $1 \leq f \leq h$,
let $T^{(f)} = \PSL_2(2^{a_f}-1)$ and let 
$y_f=\sum_{J: a(J)=a_f} y(J)$. Then the previous inequality becomes 
$$ \prod_{f=1}^h d_2(T^{(f)})^{y_f} < \sum_{f=1}^h y_f (a_f+2). $$
Since $y_f(a_f+2) \geq 5$ for each $f$, we may replace the sum by a product. We then have
\begin{equation} \label{p-is-2-ineq}
    \prod_{f=1}^h b_f(y_f) < 1, 
\end{equation}
where 
$$ b_f(y_f) = \frac{(2^{a_f-1}-1)^{y_f}}{y_f(a_f+2)}. $$
For $f=1$, we calculate
$b_1(1)= \frac{3}{5}$, $b_1(2)=\frac{9}{10}$ and $b_1(k) \geq \frac{9}{5}$ for $k \geq 3$. Moreover, for $f>1$ we have $a_f \geq 5$ and $b_f(k) \geq \frac{15}{7}$ for all $k \geq 1$. 
Thus (\ref{p-is-2-ineq}) can only hold if $h=1$ and $y_1=1$ or $2$. 
Hence $T_J \cong \PSL_2(7)$ for every $J$. Moreover, as $y(J) \geq 1$ for each $J$, 
there can be at most two minimal normal subgroups $J$ of $G$: either there are two minimal normal subgroups $J_1$, $J_2$ with $y(J_1)=y(J_2)=1$, or there is a unique minimal normal subgroup $J$ and $y(J) \leq 2$. 
As (\ref{key-ineq-bis}) now becomes
$$ \prod_J 3^{y(J)} < \sum_J \frac{5y(J)}{z(J)}, $$ 
we must have $z(J)=1$ in each of these cases. Using (\ref{mk-ry}) to find $r(J)$ in terms of $m$, we obtain (i)--(iii) as in the statement.
\end{proof}

\section{Conclusion of the proof of Theorem \ref{irred-class}}  \label{conclusion}

To complete the proof of Theorem \ref{irred-class}, we analyse the three cases listed in Theorem \ref{just-168}.
We will show that cases (i) and (ii) do not give any solutions to Question \ref{main-qn}, and that any solution given by case (iii) is one of those constructed 
in Theorem \ref{irred-constr}. We keep the previous notation, so $(G,V)$ is an irreducible solution to Question \ref{main-qn}, and $V=\F_p^n$.

From \cite{online-atlas}, we know that the absolutely irreducible representations in characteristic $2$ of the group $T \cong \PSL_2(7) \cong \GL_3(2)$ have degrees
$1$, $3$, $3$ and $8$, and that all these representations are realised over $\F_2$. Thus, in the notation of \S\ref{clifford}, we may take $E=\F_2$ and 
$W=V$. Then, again writing $S=\soc(G)$, (\ref{U-sum}) becomes
\begin{equation} \label{V-sum}
  V = \bigoplus_{i=1}^m g_i U 
\end{equation}
with $g_1=1$, where $U$ is an irreducible $\F_2[S]$-module. 

\begin{lemma}  \label{elim-tensors}
Cases (i) and (ii) of Theorem \ref{just-168} do not give any irreducible solutions to Question \ref{main-qn}.
\end{lemma}
\begin{proof}
In case (i) we have $S=J_1 \times J_2$. Noting that $r=m$, we write $J_1=T_1 \times \cdots \times T_m$ and $J_2 = T_{m+1} \times \cdots \times T_{2m}$, where each $T_k \cong T$. Using Proposition \ref{irred-tensor} to decompose $U$ as in the proof of Lemma \ref{key-ineq-lemma}, we have $U=U_{J_1} \otimes_{\F_2} U_{J_2}$, where $U_{J_h}$ is an irreducible $\F_2[J_h]$-module for $h=1$, $2$. Taking first $h=1$, we can further decompose $U_{J_1}$ as
$$	 U_{J_1} = U_{J_1,1} \otimes \cdots \otimes U_{J_1,m} $$
where $U_{J_1,k}$ is an irreducible $\F_p[T_k]$-module for $1\leq k \leq m$. Recall that, for each $i$, there are $y(J_1)$ indices $k$ such that the linear action of $T_k$ on $g_i U$ is not trivial. Since $y(J_1)=1$, we may choose the numbering of the factors $T_k$ in $J_1$ so that only $T_i$ acts non-trivially on $g_i U$. In particular, taking $i=1$ and recalling that $U=g_1 U$, it follows that $U_{J_1,1}$ is a non-trivial irreducible $\F_2[T_1]$-module, and, for $2 \leq j \leq m$, that $U_{J_1,j}=\F_2$ (the irreducible $\F_2[T_j]$-module with trivial action). Hence $\dim U_{J_1}=\dim U_{J_1,1}=3$ or $8$. Similarly, for $h=2$ we decompose $U_{J_2}$ as 
$$	  U_{J_2} = U_{J_2,1} \otimes \cdots \otimes U_{J_2,m}  $$
where $U_{J_2,k}$ is an irreducible $\F_2[T_{m+k}]$-module for $1 \leq k \leq m$. 
We choose the numbering of the factors in $J_2$ so that $T_{m+i}$ acts non-trivially on $g_i U$ for each $i$. Then $\dim U_{J_2}=\dim U_{J_2,1}=3$ or $8$. We therefore have 
$\dim U=9$ or $24$ or $64$.

In case (ii) we have $S=J=T_1 \times \cdots \times T_{2m}$ and, applying Proposition 
\ref{irred-tensor} again, we have
$$	  U = U_1 \otimes \cdots \otimes U_{2m}   $$
where exactly two of the groups $T_k$ act non-trivially on each summand $g_i U$ of $V$.
We choose the numbering of the factors in $J$ so that again these are $T_i$ and $T_{m+i}$. As before $\dim U_j=1$ if $j \neq 1$, $m+1$, and $\dim U=9$ or $24$ or $64$.

For the rest of the proof, we treat (i) and (ii) simultaneously. Since 
$S \cong T^{2m}$ and $G$ embeds in $\Aut(S) \cong \Aut(T)^{2m} \rtimes S_{2m}$ by Theorem \ref{no-CFSG}(iii), we have
$$ m \dim U = \dim V \leq v_2(|G|)  \leq 2m v_2(|\Aut(T)|) +  v_2 \bigl( (2m)! \bigr) < 10m, $$
so $\dim U =9$. Hence $U=U_1 \otimes_{\F_2} U_{m+1}$ where $U_1$ and $U_{m+1}$ are $3$-dimensional irreducible modules over $\F_2[T_1]$ and $\F_2[T_{m+1}]$ respectively. In particular, $T_1 = \Aut(U_1)$ and $T_{m+1}=\Aut(U_{m+1})$.

We claim that $U$ is a fixed-point free $\F_2[T_1 \times T_{m+1}]$-module in the sense of Proposition \ref{dir-sum}.
Take bases $e_1$, $e_2$, $e_3$ for $U_1$ and $f_1$, $f_2$, $f_3$ for $U_{m+1}$. If $x = \sum_{j,s} x_{js} e_j \otimes f_s$ is a fixed point, then applying the element of $T_1$ which permutes $e_1$, $e_2$, $e_3$ cyclically, we see that $\sum_{j,s} x_{js} e_{j+1} \otimes f_s =  \sum_{j,s} x_{js} e_j \otimes f_s$, so $x_{js}$ is independent of $j$ for each $s$. Similarly, $x_{js}$ is independent of $s$, so $x=x_{11} \sum_{j,s} e_j \otimes f_s$. However, this element must also be fixed by the element of $T_1$ with $e_1 \mapsto e_1+e_2$, $e_2 \mapsto e_2$, $e_3 \mapsto e_3$, so that $x_{11}=0$. Thus $x=0$, proving the claim.
Similarly, $g_i U$ is a fixed-point free $\F_2[T_i \times T_{i+m}]$-module for each $i$.

Applying Proposition \ref{dir-sum} repeatedly, we may view $T_1 \times T_{m+1}$ as a subgroup of $\Aff(U)$. As $T_1 \times T_{m+1}$ is subnormal in $G$, it follows from Lemma \ref{comp-factors}(i) that the stabiliser of $0_U=0_V$ in $T_1 \times T_{m+1}$ is a soluble subgroup of $2$-power index. Any such subgroup has the form $T'_1 \times T'_{m+1}$, where $T'_1$, $T'_{m+1}$ are subgroups of order $21$ in $T_1$, $T_{m+1}$. In particular, $T'_1$ contains an element conjugate to the matrix $A$ in (\ref{AB}), and, since $A-I$ is invertible, we may apply Proposition \ref{tensor-product} to conclude that $T_{m+1} \cdot 0_V=\{0_V\}$. This is a contradiction since $T_{m+1}$ is insoluble.
\end{proof}

\begin{lemma} 
If $(G,V)$ is an irreducible solution of Question \ref{main-qn} satisfying (iii) of Theorem \ref{just-168} then, up to conjugation by $\Aut(V)$, $G$ is as constructed in Theorem \ref{irred-constr}. Moreover, up to conjugation by $\Aut(V)$, $G$ is uniquely determined by the integer $r$ and the subgroup $H$ of $S_r$ in Theorem \ref{irred-constr}.
\end{lemma}
\begin{proof}
In this case we have $r=m$ and $S=J=T_1 \times \cdots \times T_m$. Since $y(J)=1$, we may assume that the $T_i$ are numbered so that $T_k$ acts non-trivially on $g_i U$ if and only if $k=i$. Arguing as in the proof of Lemma \ref{elim-tensors}, we have $\dim U =3$ or $8$. Moreover,
$$ m \dim U \leq v_2(|G|)  \leq m v_2 (|\Aut(T)|) + v_2(m!) < 5m, $$
so $\dim U=3$. By Lemma \ref{comp-factors}(i), the stabiliser $J'$ of $0_V$ in $J$ is a soluble subgroup of $2$-power index, so $J'=T'_1 \times \cdots \times T'_r$ where each $T'_i$  has index $8$ in $T_i$. It follows that the orbit of $0_V$ under $J$ has cardinality
$8^r=|V|$, so that $J$ acts transitively on $V$.

Since any irreducible $\F_2[T_i]$-module of dimension $3$ is fixed-point free, it follows from Proposition \ref{dir-sum} that 
$$    G \leq \Aff(g_1 U) \times \cdots \times \Aff(g_r U). $$
Conjugating $G$ by an element of $\Aut(V)$, we may assume the standard basis of $V$ is obtained by 
concatenating bases of $g_1 U$, \ldots, $g_r U$. Thus $J$ consists of block matrices in $\GL_{3r+1}(2)$ of the form
\begin{equation} 
 \left( \begin{array}{cccc|c} M_1 & 0 & \cdots & 0 & v_1 \\ 
                                                   0 & M_2 & \cdots & 0 &  v_2 \\ 
                                                   \vdots & \vdots &   & \vdots & \vdots \\ 
                                                  0 & 0 & \cdots & M_r & v_r \\ \hline
                                                   0 & 0 & \cdots & 0 & 1
                             \end{array} \right),
\end{equation}                
with $v_i \in g_i U$. Moreover, we may suppose that the basis of each $g_i U$ is chosen so that the stabiliser $T'_i$ of $0_V$ in $T_i$ is generated 
by the matrices in $T_i$ with $M_i=A$ and $M_i=B$, where $A$ and $B$ are as in (\ref{AB}). Then, by Lemma \ref{unique-168}, $v_i=\psi(M_i)$, so the elements of $J$ are as in (\ref{M-matrix}).

We embed the symmetric group $S_r$ into $\Aut(V)$ by taking each $\pi \in S_r$ to the block
permutation matrix $P(\pi)$ as in the proof of Theorem \ref{irred-constr} in  \S\ref{example-sec}. Then $S_r$ normalises $J$, and we can form the semidirect product $J' \rtimes S_r$ inside $\Aff(V)$. 

Since $J$ acts transitively on $V$, we have $JG'=G$ and $G/J \cong G'/J'$, where $G'$ is the stabiliser of $0_V$ in $G$. As $G' \cdot 0_V=\{0_V\}$ and $G'$ normalises $J$, it follows from Proposition \ref{J-norm} below that $G' \leq J' \rtimes S_r$. Thus $G'=J' \rtimes H$ for some subgroup $H$ of $S_r$. This subgroup is uniquely determined by $G$ once we have fixed the basis of $V$ and therefore fixed the embedding of $S_r$ into $\Aut(V)$. Then $G=JG'=J \rtimes H$. Also, $H$ must be soluble since $G'$ is soluble. As $G$ acts transitively on $\{T_1, \ldots, T_r\}$, but $J$ acts trivially on this set, it follows that $H$ is transitive as a subgroup of $S_r$. Hence $H$ satisfies the conditions in Theorem \ref{irred-constr}, and $G$ is exactly the subgroup $T^r \rtimes H$ of $\Aff(V)$ constructed from $H$ there. In particular, the group $G$ in the statement of Theorem \ref{irred-constr} is uniquely determined, up to conjugation by $\Aut(V)$, by the given subgroup $H$ of $S_r$.
\end{proof}

\begin{proposition} \label{J-norm}
With the above notation, 
$$  \{X \in \Aff(V) : X \cdot 0_V =0_V \mbox{ and } XJX^{-1} = J\}= J' \rtimes S_r.  $$ 
\end{proposition}
\begin{proof}
We first claim that the linear span of the invertible block diagonal matrices in $\GL_{3r}(2)$ is the full algebra of block diagonal matrices. Let $E^{(1)}$, $E^{(2)}$, $E^{(3)}$ be three (invertible) matrices in $\GL_3(2)$ whose sum is the zero matrix. One possible choice is 
$$ E^{(1)}= \left( \begin{array}{ccc} 1 & 1 & 0 \\0 & 1 & 1 \\ 1 & 0 & 0 \end{array} \right), \quad
 E^{(2)}= \left( \begin{array}{ccc} 0 & 1 & 0 \\0 & 0 & 1 \\ 1 & 0 & 1 \end{array} \right), \quad
 E^{(3)}= \left( \begin{array}{ccc} 1 &  0& 0 \\0 & 1 & 0 \\ 0 & 0 & 1 \end{array} \right). $$
Given $1 \leq i \leq r$ and $M \in \GL_3(2)$, let $K^{(1)}$, $K^{(2)}$, $K^{(3)}$ be the invertible block diagonal matrices in $\GL_{3r}(2)$ such that $K^{(j)}$ has $M$ in the $i$th block and $E^{(j)}$ in each of the other blocks.
Then $K^{(1)}+K^{(2)}+K^{(3)}$ has $3M=M$ in the $i$th block and $0$ in all other blocks. It therefore suffices to observe that the full ring of $3 \times 3$ matrices over $\F_2$ is spanned by its unit group 
$T=\GL_3(2)$. Indeed, it is easy to see that the $6$ elementary matrices and the $6$ permutation matrices in $T$ span the matrix ring. This proves the claim.
  
Clearly $J' \rtimes S_r$ fixes $0_V$ and normalises $J$. Conversely, suppose that $X \cdot 0_V=0_V$ and $XJX^{-1}=J$. Write $X$ as a block matrix
$$ X = \left( \begin{array}{cccc|c} P_{11} & P_{12} & \cdots & P_{1r} & 0 \\ 
                                                   P_{21} & P_{22} & \cdots & P_{2r} &  0 \\ 
                                                   \vdots & \vdots &   & \vdots & \vdots \\ 
                                                   P_{r1} & P_{r2} & \cdots & P_{rr} & 0 \\ \hline
                                                   0 & 0 & \cdots & 0 & 1
                             \end{array} \right). $$
Similarly, let $X^{-1}$ be written as a block matrix with the matrix $Q_{ij}$ in row $i$ and column $j$.
Thus we have the equalities of $3\times 3$ matrices 
$$ \sum_{j=1}^r P_{ij} Q_{jk} = \delta_{ik} I = \sum_{j=1}^r Q_{ij} P_{jk}
   \mbox{ for } 1 \leq i,k \leq r. $$                             
Now let $Y \in J$ have diagonal blocks $M_1$, \ldots $M_r$ as in (\ref{M-matrix}). 
As $XYX^{-1} \in J$, we have 
 $$ XYX^{-1} = \left( \begin{array}{cccc|c} N_1 & 0 & \cdots & 0 & \psi(N_1) \\ 
                                                   0 & N_2 & \cdots & 0 &  \psi(N_2) \\ 
                                                   \vdots & \vdots &   & \vdots & \vdots \\
                                                   0 & 0 & \cdots & N_r & \psi(N_r) \\ \hline
                                                   0 & 0 & \cdots & 0 & 1
                             \end{array} \right) $$                          
for some $N_1$, \ldots, $N_r \in \GL_3(2)$. It follows that 
\begin{equation} \label{conj1}
   \sum_{j=1}^r P_{ij}M_j Q_{jk} = N_i \delta_{ik} 
\end{equation}
and
\begin{equation} \label{conj2}
   \sum_{j=1}^r P_{ij} \psi(M_j) = \psi(N_i). 
\end{equation}
Let us write $\overline{X}$ for the projection of $X$ into $\Aut(V)=\GL_{3r}(2)$. 
Then (\ref{conj1}) shows that conjugating an invertible block diagonal matrix in $\GL_{3r}(2)$ by $\overline{X}$ gives another invertible block diagonal matrix. It then follows from the claim at the start of the proof that conjugating any block diagonal matrix (not necessarily invertible) by $\overline{X}$ gives a block diagonal matrix. Thus, given any $3 \times 3$ matrices $M_1$, \ldots, $M_r$ over $\F_2$, there exist matrices $N_1$, \ldots, $N_r$ so that (\ref{conj1}) holds.

Let $i \in \{1, 2, \ldots, r\}$. Since $X$ is invertible, there must be some $h$ with $P_{ih} \neq 0$.  Consider the block diagonal matrix $M$ with an arbitrary matrix $M_h$ in the $h$-th position, and all other blocks $0$. Then the $ij$-th block of $PMQ$ is
$P_{ih} M_h Q_{hj}$. This must be $0$ if $j \neq i$. Since this holds for arbitrary $M_h$, it follows that $Q_{hj}=0$ for all $j \neq i$. Since $\sum_j Q_{hj} P_{jh}=I$, we find that $Q_{hi}P_{ih}=I$, so $Q_{hi}$ and $P_{ih}$ are invertible. As $Q_{hi} \neq 0$, a similar argument shows that $P_{ij}=0$ for all $j \neq h$. Thus for each $i$, there is a unique $h$ with $P_{ih} \neq 0$. Similarly, for each $h$, there is a unique $i$ with $P_{hi} \neq 0$. Thus there is a permutation $\pi \in S_r$ such that $P_{ij} \neq 0$ if and only if $i= \pi(j)$. Setting $Z_j=P_{\pi(j)j} \in \GL_3(2)$, we find that 
$$ P_{ij} = \begin{cases} Z_j & \mbox{ if } i = \pi(j), \\
                                        0 & \mbox{ otherwise.}  \end{cases} $$
Thus 
$$ Q_{jk} = \begin{cases} Z_j^{-1} & \mbox{ if } k = \pi(j), \\
                                        0 & \mbox{ otherwise.}  \end{cases} $$
Then (\ref{conj1}) and (\ref{conj2}) simplify to 
$$ Z_j M_j Z_j^{-1} = N_{\pi(j)}, \qquad Z_j \psi(M_j) =\psi(N_{\pi(j}). $$
If we now take $M$ to be in $J'$, so $M_j \in T'$ and $\psi(M_j)=0$ for all $j$, 
then we have $\psi(N_i)=0$ and $N_i \in T'$ for all $i$. Hence each $Z_j$ lies in the normaliser of $T'$ in $T$. Since $T'$ is itself the normaliser of a Sylow $7$-subgroup of $T$, this gives $Z_j \in T'$. The block diagonal matrix $Z$ with blocks $Z_j$ then lies in $J'$, so $X= Z^{-1} P(\pi) \in J' \rtimes S_r$.
\end{proof}

\bibliography{sol-insol-refs}
\end{document}